\documentclass[11pt,a4paper]{amsart}
\usepackage{amsmath, amssymb, amsthm}
\usepackage[hidelinks]{hyperref}
\usepackage{enumitem}

\usepackage{tikz}
\usetikzlibrary{decorations.markings}
\usepackage[all]{xy}

\numberwithin{equation}{section}

\newcommand{\cS}{\mathcal S}
\newcommand{\calR}{{\mathcal R}}

\newcommand{\cE}{{\mathcal E}}
\newcommand{\calD}{\mathcal D}
\newcommand{\calL}{{\mathcal L}}
\newcommand{\cU}{{\mathcal U}}
\newcommand{\cZ}{{\mathcal Z}}
\newcommand{\cO}{\mathcal O}
\newcommand{\cK}{\mathcal K}
\newcommand{\cM}{\mathcal M}
\newcommand{\cQ}{\mathcal Q}
\newcommand{\cB}{\mathcal B}
\newcommand{\cC}{\mathcal C}
\newcommand{\cF}{\mathcal F}

\newcommand{\bbC}{{\mathbb C}}
\newcommand{\bbN}{{\mathbb N}}
\newcommand{\bbU}{{\mathbb U}}
\newcommand{\bbZ}{{\mathbb Z}}
\newcommand{\bbQ}{\mathbb Q}

\newcommand{\sfD}{\mathsf D}
\newcommand{\sfX}{\mathsf X}
\newcommand{\sfY}{\mathsf Y}

\newcommand{\Ext}{{\rm Ext}}
\DeclareMathOperator{\SL}{SL}
\newcommand{\e}{\varepsilon}
\newcommand{\cstar}{\ensuremath{\mathrm{C}^*}}
\DeclareMathOperator{\id}{id}
\DeclareMathOperator{\Sep}{Sep} 
\newcommand{\oneHS}[1]{\|#1\|_{\textrm{HS},1}}
\newcommand{\kHS}[1]{\|#1\|_{\textrm{HS},K}}

\newtheorem{theorem}{Theorem}[section]
\newtheorem*{theorem*}{Theorem}
\newtheorem{proposition}[theorem]{Proposition}
\newtheorem*{proposition*}{Proposition}
\newtheorem{lemma}[theorem]{Lemma}
\newtheorem*{lemma*}{Lemma}
\newtheorem{corollary}[theorem]{Corollary}
\newtheorem*{corollary*}{Corollar}

\theoremstyle{definition}
\newtheorem{definition}[theorem]{Definition}
\newtheorem*{definition*}{Definition}

\newtheorem*{conjecture*}{Conjecture}
\newtheorem*{notation}{Notation}
\newtheorem{theoremi}{Theorem}

\theoremstyle{remark}
\newtheorem{example}[theorem]{Example}
\newtheorem*{example*}{Example}
\newtheorem{remark}[theorem]{Remark}
\newtheorem*{remark*}{Remark}

\newtheorem*{note*}{Note}
\newtheorem*{question*}{Question}

\author{Ilijas Farah}

\address{Department of Mathematics and Statistics, York University, 4700 Keele \linebreak \phantom{---} Street, Toronto, Ontario, Canada, M3J 1P3} 
\address{Matemati\v cki Institut SANU, Kneza Mihaila 36, 11\,000 Beograd, p.p. 367,\linebreak\phantom{---} Serbia}
\email{ifarah@yorku.ca}
\urladdr{https://ifarah.mathstats.yorku.ca/}
\urladdr{https://orcid.org/0000-0001-7703-6931 (ORCID iD)}

\author{G{\'a}bor Szab{\'o}}
\address{Department of Mathematics, KU Leuven, Celestijnenlaan 200B, box 2400,\linebreak \phantom{---} 3001 Leuven, Belgium}
\email{gabor.szabo@kuleuven.be}

\thanks{I.F. partially supported by NSERC}
\thanks{G.S. funded by the European Union.
Views and opinions expressed are those of the authors only and do not necessarily reflect those of the European Union or the European Research Council.
Neither the EU nor the ERC can be held responsible for them.}

\title[Coronas \& strongly self-absorbing \cstar-algebras]{Coronas and strongly self-absorbing \cstar-algebras}
\begin{document}
\maketitle
\begin{abstract}
Let $\calD$ be a strongly self-absorbing \cstar-algebra.
Given any separable \cstar-algebra $A$, our two main results assert the following.
If $A$ is $\calD$-stable and non-unital, then the corona algebra of $A$ is $\calD$-saturated, i.e., $\calD$ embeds unitally into the relative commutant of every separable \cstar-subalgebra. 
Conversely, assuming that the stable corona of $A$ is separably $\calD$-stable, we prove that $A$ is $\calD$-stable.
This generalizes recent work by the first-named author on the structure of the Calkin algebra.
As an immediate corollary, it follows that the multiplier algebra of a separable $\calD$-stable \cstar-algebra is separably $\calD$-stable.
Appropriate versions of the aforementioned results are also obtained when $A$ is not necessarily separable.
The article ends with some non-trivial applications.
\end{abstract}

\setcounter{tocdepth}{1}
\tableofcontents

\section*{Introduction}

Within the structure and classification theory of \cstar-algebras, the so-called strongly self-absorbing \cstar-algebras have historically emerged as important cornerstone objects.
This was initially the case by example:
The special Cuntz algebras $\mathcal O_\infty$ and $\mathcal O_2$ were shown by Kirchberg--Phillips \cite{KirchbergPhillips00} to be, in an appropriate sense, the initial and final object (respectively) in the category of separable simple unital nuclear purely infinite \cstar-algebras (the so-called \emph{Kirchberg algebras}).
Although it was not originally phrased in that (not yet existing) framework at the time, Phillips' proof \cite{Phillips00} of the Kirchberg--Phillips classification theorem \cite{Kirchberg95icm, KirchbergC} made crucial use of the properties of $\mathcal O_\infty$ that make it a strongly self-absorbing \cstar-algebra from today's point of view.
Towards the end of the 1990s, Jiang and Su \cite{JiangSu99} constructed the \cstar-algebra $\mathcal Z$ now named after them, intended as a stably finite analog of $\mathcal O_\infty$, and showed that it is strongly self-absorbing.

Finally, the behavior of these interesting examples was unified in the abstract theory of strongly self-absorbing \cstar-algebras (recalled in the preliminary section) in the mid 2000s, by Toms--Winter \cite{ToWi:Strongly} and Kirchberg \cite{Kirc:Central}.
Let $\calD$ be a strongly self-absorbing \cstar-algebra.
One of the most appealing aspects of the associated abstract theory is that given any separable \cstar-algebra~$A$, one can characterize the existence of an isomorphism $A\cong A\otimes\calD$ (this is called \emph{$\calD$-stability}) as an internal local approximation property of $A$ via a so-called McDuff-type criterion, inspired by a related result \cite{McDuff70} for von Neumann algebras that served as the original inspiration.
This has the advantage of making $\calD$-stability a reasonably verifiable condition in many contexts and moreover implies that $\calD$-stability is preserved under various elementary constructions.
Subsequently, strongly self-absorbing \cstar-algebras had been recognized both as interesting test cases for hard open problems (see \cite{Winter17abel, Schafhauser24}) and as objects in their respective categories playing a facilitating role for certain classification problems \cite{Winter14}.
Given that $\cZ$ has been shown in \cite{winter2011strongly} to be the initial object in the category of strongly self-absorbing \cstar-algebras, the property of tensorially absorbing $\cZ$ (called \emph{Jiang--Su stability}) has a special significance.
In the history of the subject, it has emerged as a well-studied property both in examples \cite{TomsWinter13, ElliottNiu17, Kerr20, KerrSzabo20, GardellaGeffenNaryshkinVaccaro24, Naryshkin24} and abstractly within the work on the Toms--Winter conjecture \cite{ElliottToms08, Winter10, Winter12, MatuiSato12, Tikuisis14, MatuiSato14uhf, BBSTWW, CETWW21, CastillejosEvington21} (the given references are not exhaustive).
Indeed, the most satisfactory state-of-the-art classification theorem in the context of the Elliott program concerns the category of separable simple nuclear Jiang--Su stable \cstar-algebras satisfying the universal coefficient theorem \cite{GongLinNiu20, GongLinNiu20-2}; see \cite{carrion2023class} and the more elaborate discussion therein about the history of classification.

The present work is a continuation and strengthening of the recent article~\cite{farah2023calkin} by the first named author, motivated by the Brown--Douglas--Fillmore question whether the Calkin algebra can have a K-theory reversing automorphism. In this work a new connection was discovered between strongly self-absorbing \cstar-algebras and the stucture of corona algebras.
Recall that for a \cstar-algebra $A$ with associated multiplier algebra $\cM(A)$, the quotient $\cQ(A):=\cM(A)/A$ is called the \emph{corona algebra} of $A$.
Before resuming the discussion, we recall a concept frequently occurring in this article:

\begin{definition*} 
Let $B$ and $C$ be unital \cstar-algebras with $B$ separable.
We say that $C$ is \emph{$B$-saturated} if for every separable \cstar-subalgebra $A\subseteq C$, there is a unital, injective $*$-homomorphism from $B$ to $C\cap A'$.
\end{definition*}

In simplified terms, the two main insights in \cite[Theorems B+C]{farah2023calkin} can be summarized by the following statements:
If $\calD$ is strongly self-absorbing and  a separable unital \cstar-algebra $A$ is $\calD$-stable, then $\cQ(\cK\otimes A)$ is $\calD$-saturated; the Calkin algebra $\cQ(H):=\cQ(\cK)$, however, is not $\calD$-saturated.
As an immediate consequence, it followed (among other things) that the Calkin algebra is not isomorphic to the corona $\cQ(\cK\otimes\mathcal O_\infty)$, which answered a previously open question.

In this article, we generalize both of these insights to the most general setting in which they are plausible.
In order to extend the theory beyond the separable case, we work with the concept of \emph{separable $\calD$-stability} from \cite{schafhauser2020new}.
A \cstar-algebra $A$ is called separably $\calD$-stable if every separable \cstar-subalgebra of $A$ is included in a separable $\calD$-stable \cstar-subalgebra of $A$.

Our first main result (proven in Section \ref{S.D-saturation} as Theorem~\ref{T.Q-saturated.1}) gives a generalization of \cite[Theorems C]{farah2023calkin} by removing a number of prior assumptions.

\begin{theoremi} \label{T.Q-saturated}
Suppose that $\calD$ is a strongly self-absorbing \cstar-algebra and that $A$ is a $\sigma$-unital \cstar-algebra that is separably $\calD$-stable.
Then $\cQ(A)$ is $\calD$-saturated. 
\end{theoremi}

The assumption of $\sigma$-unitality is not only essential in our proof of Theorem~\ref{T.Q-saturated}, but taking a closer look at a class of non-$\sigma$-unital \cstar-algebras studied by Sakai (discussed in Example~\ref{ex:1-dim-corona}) demonstrates that there exist separably $\cZ$-stable \cstar-algebras with one-dimensional corona algebras.

Our second main result, to be understood as a powerful generalization of both main results in \cite{farah2023calkin}, provides a partial converse of the above and entails a new characterization of (separable) $\calD$-stability, also having consequences to the structure of the multiplier algebra.
It fully confirms and extends \cite[Conjecture 6.2]{farah2023calkin}.

\begin{theoremi} \label{T.main}
Suppose that $\calD$ is a strongly self-absorbing \cstar-algebra and $A$ is a $\sigma$-unital \cstar-algebra.
Then the following are equivalent: 
	\begin{enumerate}[label=\textup{(\arabic*)}]
		\item \label{T.main.A-tensorially-absorbing} $A$ is separably $\calD$-stable. 
		\item \label{T.main.Q-locally-absorbing} $\cQ(\cK\otimes A )$ is separably $\calD$-stable. 
		\item \label{T.main.M-locally-absorbing} $\cM(A)$ is separably $\calD$-stable. 
	\end{enumerate}
\end{theoremi}

Apart from Theorem \ref{T.Q-saturated} covering the implication \ref{T.main.A-tensorially-absorbing}$\Rightarrow$\ref{T.main.Q-locally-absorbing}, the most novel aspect in the proof of Theorem~\ref{T.main} concerns the implication \ref{T.main.Q-locally-absorbing}$\Rightarrow$\ref{T.main.M-locally-absorbing}.
This implication is true even without $\sigma$-unitality and arises as the consequence of a deeper principle that allows us to transfer (relatively) approximately central sequences from $\cQ(\cK\otimes A)$ to $\cM(A)$; see Definition~\ref{Def.approx-central-embeddings} and  Corollary~\ref{C.Psi} for details.
This is achieved in Sections \ref{S.Schur} and \ref{S.Gamma.A} based on studying the relative positions of certain unitary representations of property (T) groups inside the stable corona $\cQ(\cK\otimes A)$ akin to  \cite[Theorem~B]{farah2023calkin}, with a few upgrades.
Although we currently have no definitive evidence of this claim, we believe that a further generalization of implication \ref{T.main.Q-locally-absorbing}$\Rightarrow$\ref{T.main.A-tensorially-absorbing} cannot be expected in any reasonable generality if one only makes a structural assumption about $\cQ(A)$ without also assuming that $A$ is stable or simple.
This is demonstrated by various easy examples such as $A:=\bbC\oplus(\cK\otimes\calD)$.

The last Section \ref{S.Applications} is concerned with applications of our main results.
We present an application of Theorem~\ref{T.Q-saturated}, whereby the Calkin algebra does not have the same first-order theory as any nuclear \cstar-algebra (i.e., is not elementarily equivalent to a nuclear \cstar-algebra); see Theorem~\ref{T.Calkin}.
Also, if $A$ is any $\sigma$-unital \cstar-algebra and $\cQ(A)$ is isomorphic to the Calkin algebra, then $A$ cannot be separably $\cZ$-stable (Corollary~\ref{C.JiangSu}). 
Another application is a partial positive result to a question attributed to Sakai in~\cite{Ell:Derivations}, asking whether non-unital simple separable \cstar-algebras are isomorphic if and only if their coronas are isomorphic (see \cite[Question~4.19]{farah2022corona}).
It is an immediate consequence of Theorem~\ref{T.main} that coronas of stabilizations of strongly self-absorbing \cstar-algebras are isomorphic if and only if the algebras are isomorphic.
Next, we observe that multiplier algebras of separably $\cZ$-stable \cstar-algebras have strict comparison, which is related to earlier results in the literature obtained under comparably more restrictive assumptions.
Lastly, we can give a partial positive answer to \cite[Question 5.17]{carrion2023class} and give an improved version of the recent $\cZ$-stable $\textit{KK}$-uniqueness theorem therein; see Theorem~\ref{T.Z-stable-KK} for details.
	While we have novel applications to the model theory of coronas, we would like to emphasize that we do not develop any novel methodology of model-theoretic or set-theoretic nature.

\subsection*{Acknowledgments.} 
We are indebted to James Gabe, Chris Schafhauser and Stuart White for remarks made during the Memorial Conference in honour of Eberhard Kirchberg in M\"unster in July 2023 that led to results presented here and to Hannes Thiel for pointing out some corollaries to our results.
I.F.\ is grateful to Philip Spencer for finding a crazy bug in (otherwise excellent) \TeX studio and suggesting a fix. 
G.S.\ would like to thank Matteo Pagliero for some inspriring discussions on the topic of this article.

G.S.\ was supported by research project C14/19/088 funded by the research council of KU Leuven, research project G085020N funded by the Research Foundation Flanders (FWO), and the European Research Council under the European Union's Horizon Europe research and innovation programme (ERC grant AMEN--101124789).

For the purpose of open access, the authors have applied a CC BY public copyright license to any author accepted manuscript version arising from this submission.
 
 
\section{Preliminaries} \label{S.Terminology}

For what follows below, we fix an arbitrary \cstar-algebra $A$, on which we may impose further assumptions when necessary.
Consider the multiplier algebra $\cM(A)$ of $A$ and its corona $\cQ(A):=\cM(A)/A$.

\subsection{Clubs and reflection} \label{S.Reflection}
Reflection is the phenomenon (well-studied in set theory; see e.g., the introduction to \cite{foreman2009handbook} or \cite{magidor1994does}) when properties of an uncountable (more generally nonseparable) structure reflect to its small (for example,  countable or separable) substructures.
In \cstar-algebras reflection is usually obtained by Blackadar's method  (see the discussion of  separably inheritable properties in  \cite[\S II.8.5]{Black:Operator}) and in a more general case by the Downward L\" owenheim--Skolem theorem from model theory (see e.g.,  \cite[Corollary~2.1.4]{ChaKe} for the general statement, \cite[Proposition~7.2]{BYBHU} for the continuous logic,   and \cite[\S 2.6]{Muenster}  or \cite[Appendix D]{Fa:STCstar} for \cstar-algebras).
An in-depth discussion of reflection phenomena for \cstar-algebras can be found in \cite[\S 7.3 and \S 7.4]{Fa:STCstar}. 
To give a more precise definition of the reflection that we will be using, we need the notion of an elementary submodel.
If $A\subseteq B$ are \cstar-algebras, then $A$ is an \emph{elementary submodel} of $B$ if for every first-order formula $\varphi(\bar x)$ and every tuple $\bar a$ in $A$ of the same sort as $\bar x$ the evaluations of $\varphi(\bar a)$ in $A$ and in $B$ agree  (\cite[Definition~2.3.3]{Muenster} or \cite[\S 7.1, also \S D.1]{Fa:STCstar}). A familiar example is provided by \L o\' s's Theorem (e.g., \cite[Theorem~2.3.1]{Muenster}):\ the diagonal copy of $A$ in its ultrapower $A_\cU$ is an elementary submodel. 
Properties captured by the first-order theory are called \emph{axiomatizable} or \emph{elementary}. See \cite[Theorem 2.5.1]{Muenster} for numerous examples. 

An elementary submodel inherits all axiomatizable properties of the original algebra 
as well as many non-axiomatizable properties (such as nuclearity and other properties definable by uniform families of formulas, see \cite[Theorem~5.7.3]{Muenster}).
The following definitions are standard (\cite[Definition~6.2.6, Example~6.2.5]{Fa:STCstar}).

\begin{definition}\label{Def.Sep}
For a \cstar-algebra $C$ let $\Sep(C)$ denote the poset of its separable \cstar-subalgebras ordered by the inclusion.
A subset of $\Sep(C)$ that is cofinal and $\sigma$-closed (i.e., closed under taking direct limits of increasing sequences) is, for historical reasons, (misleadingly) known as a \emph{closed and unbounded} set or a \emph{club}. 
\end{definition}

By the Downwards L\"owenheim--Skolem theorem (as stated in \cite[Theorem~7.19]{Fa:STCstar}), for every nonseparable \cstar-algebra, its separable elementary submodels form a  club.
This is the simplest form of reflection to separable substructures. 
An important (yet not difficult to prove) feature of clubs is:

\begin{lemma}[see {\cite[Proposition~6.2.9]{Fa:STCstar}}] \label{Lem.clubs}
Let $C$ be \cstar-algebra.
Any intersection of countably many clubs in $\Sep(C)$ is a club.
\end{lemma} 

In other words, subsets of $\Sep(C)$ that include a club form a $\sigma$-filter. 
The following reflection property of clubs requires a proof.

\begin{lemma}\label{Lem.club.restriction}
Suppose that $A\subseteq B$ are nonseparable \cstar-algebras. If $\cC$ includes a club in $\Sep(B)$, then $\{A\cap C\mid C\in \cC\}$ includes a club in $\Sep(A)$. 
\end{lemma}
\begin{proof} We will use a well-known fact about clubs in discrete structures. If~$\sfX$ is an uncountable set then~$[\sfX]^{\aleph_0}$ is the poset of all countable subsets of~$\sfX$.
A club in this poset is a cofinal family closed under countable increasing unions.
By a corollary to Kueker's characterization of clubs in $[\sfX]^{\aleph_0}$ (\cite[Corollary~6.4.2]{Fa:STCstar}), if $\sfY\subseteq \sfX$ are uncountable sets and $\cE$ is a club in $[\sfX]^{\aleph_0}$, then $\{p\cap \sfY\mid p\in \cE\}$ includes a club in $[\sfY]^{\aleph_0}$. 
	
	Assume $A\subseteq B$ are nonseparable \cstar-algebras. 
	By \cite[Proposition~7.2.7]{Fa:STCstar} there is a dense subset $\sfD_B\subseteq B$ such that the set of relatively closed countable subsets of $\sfD_B$,  
	$\{p\in [\sfD_B]^{\aleph_0}\mid \overline p\cap \sfD_B=p\}$,
	includes a club in $[\sfD_B]^{\aleph_0}$. This implies  that, for any $\cC\subseteq \Sep(B)$, $\cC$ includes a club in $\Sep(B)$ if and only if   the set $\{C\cap \sfD_B\mid C\in \cC\}$ includes a club in $[\sfD_B]^{\aleph_0}$ (\cite[Corollary~7.2.8]{Fa:STCstar}).
	Again by \cite[Proposition~7.2.7]{Fa:STCstar} fix a dense subset $\sfD_A$ of $A$ such that the set of countable relatively closed subsets of~$\sfD_A$ includes a club. Then $\sfD:=\sfD_A\cup \sfD_B$ also has the property that the family of its countable relatively closed subsets includes a club. 
	
	Finally, consider a subset $\cC$ of $\Sep(B)$ that includes a club. By Lemma~\ref{Lem.clubs}, $\{C\cap \sfD \mid C\in\cC\}$ includes a club in $[\sfD]^{\aleph_0}$.
	As pointed out above,  \cite[Corollary~6.4.2]{Fa:STCstar} implies that $\{C\cap \sfD_A\mid C\in \cC\}$ includes a club in $[\sfD_A]^{\aleph_0}$.
	Again by  \cite[Corollary 7.2.8]{Fa:STCstar}, $\{C\cap A\mid C\in \cC\}$ includes a club, as required. 
\end{proof}

Needless to say, this proof shows that  Lemma~\ref{Lem.club.restriction} is true for arbitrary metric structures. 

\subsection{Strongly self-absorbing \cstar-algebras and separable absorption}

A separable unital \cstar-algebra $\calD$ is said to be \emph{strongly self-absorbing} if $\calD\ncong\mathbb{C}$ and there exists an isomorphism $\calD\to\calD\otimes \calD$ that is approximately unitarily equivalent to the embedding $d\mapsto d\otimes 1_{\calD}$.
All strongly self-absorbing \cstar-algebras are nuclear and simple, and they satisfy $\calD\cong \calD^{\otimes\infty}$; see \cite{ToWi:Strongly}. 
By \cite[Remark~3.3]{winter2011strongly}, it automatically follows that $\calD$ is $K_1$-injective and this property is redundant despite being explicitly assumed in early literature about strongly self-absorbing \cstar-algebras.

The known strongly self-absorbing \cstar-algebras are the Jiang--Su algebra $\cZ$ \cite{JiangSu99}, UHF algebras of infinite type, the Cuntz algebra $\cO_\infty$, tensor products of UHF algebras of infinite type with $\cO_\infty$, and the Cuntz algebra $\cO_2$ \cite{Cuntz77}.
If all separable nuclear \cstar-algebras satisfy the Universal Coefficient Theorem then there are no other strongly self-absorbing \cstar-algebras (\cite[Corollary~6.7]{tikuisis2017quasidiagonality})

We recall the common characterizations of $\calD$-stability for separable \cstar-algebras.
In the following $A_\infty$ denotes the (asymptotic) sequence algebra $\ell_\infty(A)/c_0(A)$ and~$A^\perp$ denotes the two-sided annihilator of $A$ in the ambient \cstar-algebra.

\begin{theorem} \label{Thm.Cstar-McDuff}
Let $A$ be a separable \cstar-algebra and $\calD$ a strongly self-absorbing \cstar-algebra.
The following are equivalent:
	\begin{enumerate}[label=\textup{(\arabic*)}]
	\item $A\cong A\otimes\calD$. \label{Thm.Cstar-McDuff:1}
	\item There exists a unital $*$-homomorphism $\calD\to (A_\infty\cap A')/(A_\infty\cap A^\perp)$.\label{Thm.Cstar-McDuff:2}
	\item There exists a sequence of unital $*$-homomorphisms $\pi_n: \calD\to \cM(A)$ with $\|[\pi_n(d),a]\|\to 0$ for all $d\in\calD$ and $a\in A$. \label{Thm.Cstar-McDuff:3}
	\end{enumerate}
\end{theorem}
\begin{proof}
The equivalence \ref{Thm.Cstar-McDuff:1}$\Leftrightarrow$\ref{Thm.Cstar-McDuff:2} can be deduced as the consequence of either \cite[Theorem 2.3]{ToWi:Strongly} or of \cite[Proposition 4.4]{Kirc:Central}.
Since $\calD\cong\calD^{\otimes\infty}$, it is immediate to verify the implication \ref{Thm.Cstar-McDuff:1}$\Rightarrow$\ref{Thm.Cstar-McDuff:3} by hand.
The implication \ref{Thm.Cstar-McDuff:3}$\Rightarrow$\ref{Thm.Cstar-McDuff:1} follows from \cite[Proposition 1.5]{ToWi:Strongly} and \cite[Theorem 7.2.2]{Ror:Classification}.
\end{proof}

The following concept has been circulating within the \cstar-community for a long time under various similar names.
For instance, for unital \cstar-algebras it has been called ``potentially $\calD$-absorbing'' in \cite{FaHaRoTi:Relative}.
In the stated generality, we use the name as coined by Schafhauser in \cite[Definition 1.4]{Schafhauser2020}:

\begin{definition} \label{Def.sep-D-stable}
Let $A$ be a \cstar-algebra and $\mathcal D$ a strongly self-absorbing \cstar-algebra.
We say that $A$ is \emph{separably $\mathcal D$-stable} if for every separable \cstar-subalgebra $C_0\subseteq A$, there exists a separable \cstar-algebra $C_1\subseteq A$ containing $C_0$ and $C_1\cong C_1\otimes\mathcal D$. 
\end{definition}

Given that the category of separable $\calD$-stable \cstar-algebras is closed under direct limits (\cite[Corollary 3.4]{ToWi:Strongly}), we get:

\begin{proposition} \label{Prop.ssa-clubs}
A \cstar-algebra $A$ is separably $\calD$-stable if and only if the separable $\calD$-stable \cstar-subalgebras of $A$ form a club in $\Sep(A)$.
\end{proposition}

Notice that when a \cstar-algebra $A$ is genuinely $\calD$-stable, say via an isomorphism $\varphi: A\otimes\calD\to A$, then the collection $\{ \varphi(C\otimes\calD) \mid C\in\Sep(A)\}$ forms a club of separable $\calD$-stable subalgebras of $A$, so $A$ is separably $\calD$-stable.
Separable $\mathcal D$-stability is clearly equivalent to tensorial $\mathcal D$-absorption when $A$ is separable, but is weaker and more common than $\mathcal D$-absorption for nonseparable \cstar-algebras.
As a matter of fact, coronas of $\sigma$-unital \cstar-algebras and nontrivial ultraproducts are tensorially indecomposable, and are therefore not $\calD$-absorbing (\cite{Gha:SAW*}, see also \cite[Theorem~15.4.5]{Fa:STCstar}).

Separable $\calD$-stability is a first-order property of a \cstar-algebra (in other words, it is axiomatizable; see \cite[\S 2.4]{Muenster}) which, together with \L o\'s's Theorem, explains why it is preserved by ultraproducts.
We shall study the concept of separable $\calD$-stability further and put it into a broader context that is meaningful for not necessarily strongly self-absorbing \cstar-algebras.

In the following we will also appeal to the reduced product construction related to a filter on a set.
Given a filter $\cF$ on a set~$I$, the reduced product $C_\cF$ is defined in e.g., \cite[Definition 16.2.1]{Fa:STCstar} (see also the remark about dual filters following this definition). 
Upon applying this construction to the Fréchet filter on $\bbN$ (usually denoted by the symbol $\infty$), one recovers the asymptotic sequence algebra.

\begin{proposition} \label{P.approx-central-embeddings}
Let $B$ and $C$ be \cstar-algebras such that $B$ is separable and unital.
The following are equivalent. 
	\begin{enumerate}[label=\textup{(\arabic*)}]
	\item \label{P.approx-central-embeddings.1}
	For every separable \cstar-subalgebra $A\subseteq C_\infty$, there exists a unital $*$-ho\-mo\-mor\-phism from $B$ to $(C_\infty\cap A')/(C_\infty\cap A^\perp)$.
	\item \label{P.approx-central-embeddings.2}
	For every separable \cstar-subalgebra $A\subseteq C$, there exists a unital $*$-ho\-mo\-mor\-phism from $B$ to $(C_\infty\cap A')/(C_\infty\cap A^\perp)$.\footnote{Here $A$ is identified with a \cstar-subalgebra of the diagonal copy of $C$ inside $C_\infty$.}
	\item \label{P.approx-central-embeddings.3}
	There exists some non-empty set $I$ and a filter $\cF$ on $I$ such that for every separable \cstar-subalgebra $A\subseteq C$, there exists a unital $*$-homomorphism from $B$ to $(C_\cF\cap A')/(C_\cF\cap A^\perp)$.
	\item \label{P.approx-central-embeddings.4}
	For every $\varepsilon>0$, every self-adjoint finite subset $F\subset C$ and every finite-dimensional, unital,  self-adjoint subspace $S\subset B$, there exists a $*$-linear map $\kappa: S^2 \to C$ such that for all $c\in F$ and $b, b_1, b_2$ in $S$ the following conditions hold. 
		\begin{itemize}
		\item $\|\kappa(b)\|\leq (1+\varepsilon)\|b\|$
		\item $\| [\kappa(b),c]\|\leq \|b\|\varepsilon$
		\item $\|(\kappa(b_1)\kappa(b_2)-\kappa(b_1b_2))c\|\leq \|b_1\|\|b_2\|\varepsilon$
		\item $\|(1-\kappa(1))c\|\leq\varepsilon$.
		\end{itemize}
	\end{enumerate}
\end{proposition}

\begin{definition} \label{Def.approx-central-embeddings}
Let $B$ and $C$ be \cstar-algebras such that $B$ is separable and unital.
We say that \emph{$C$ admits large approximately central maps from $B$} if the equivalent conditions in Proposition~\ref{P.approx-central-embeddings} are satisfied.
\end{definition}

\begin{proof}[Proof of Proposition~\ref{P.approx-central-embeddings}]
The implications \ref{P.approx-central-embeddings.1}$\Rightarrow$\ref{P.approx-central-embeddings.2}$\Rightarrow$\ref{P.approx-central-embeddings.3} are trivial.

\ref{P.approx-central-embeddings.3}$\Rightarrow$\ref{P.approx-central-embeddings.4}:
Let the triple $(\varepsilon, F, S)$ be given.
Let $A\subseteq C$ be the \cstar-algebra generated by $F$.
By assumption, there exist filter $\cF$ on a set $I$ and a unital $*$-homomorphism
\[
\psi: B \to (C_\cF\cap A')/(C_\cF\cap A^\perp).
\]
We can view the codomain as a quotient of
\[
E=\{ (x_i)_i\in \ell^\infty(I, C) \mid \lim_{i\to\cF} \| [x_i, a] \|=0 \text{ for all } a\in A.\}
\]
Let $\pi: E\to (C_\cF\cap A')/(C_\cF\cap A^\perp)$ be the quotient map.
By basic linear algebra and by applying the axiom of choice to a Hamel basis consisting of self-adjoint elements in $B$,  we may find a (not necessarily continuous) $*$-linear map $\varphi_0: B\to E$ with $\pi\circ\varphi_0=\psi$.
If $(e_\lambda)_\lambda$ is any increasing approximate unit of the ideal $\ker(\pi)$ in $E$, then it follows by the characterization of quotient norms of \cstar-algebras that
\[
\lim_{\lambda} \|(1-e_\lambda)\varphi_0(b)\|=\|\psi(b)\|\leq\|b\|,\quad b\in B.
\]
Using the fact that $S$ is a finite-dimensional normed vector space and hence has a compact unit ball, we may choose (as a consequence of Dini's theorem) an index $\lambda$ such that the map
\[
\varphi=(1-e_\lambda)\cdot\varphi_0(\_)\cdot(1-e_\lambda): B\to E
\]
has operator norm at most $1+\varepsilon$ when restricted to $S$.
Since $E\subseteq\ell^\infty(I, C)$, we may represent $\varphi$ by a family of $*$-linear maps $\kappa_i: B\to C$ for $i\in I$.
Then there is a set $I_0\in \cF$ such that $\|\kappa_j(b)\|\leq (1+\varepsilon)\|b\|$ for all $b\in S$ and $j\in I_0$, and furthermore $[(\kappa_i(b))_i]\in C_\cF\cap A'$ and $[(\kappa_i(b))_i]+(C_\cF\cap A^\perp) = \psi(b)$ for all $b\in S$.
The former condition implies that for all $b\in S$ and $a\in A$ we have $\limsup_{i\to\cF} \| [\kappa_i(b), a] \|=0$. 
The fact that $\psi$ is a unital map translates to $\limsup_{i\to\cF} \kappa_i(1)a=a$ for all $a\in A$.
The fact that $\psi$ is multiplicative translates to $\limsup_{i\to\cF} \| (\kappa_i(b_1)\kappa_i(b_2)-\kappa_i(b_1b_2))a \| = 0$ for all $a\in A$ and $b_1,b_2\in S$.
Putting all this together, we can see from $A\supset F$ that there exists some element $i\in I_0$ such that the $*$-linear map $\kappa=\kappa_i|_{S^2}$ satisfies the desired properties for the triple $(\varepsilon, F, S)$.

\ref{P.approx-central-embeddings.4}$\Rightarrow$\ref{P.approx-central-embeddings.1}:
Throughout the argument that follows, we fix an increasing sequence of finite-dimensional subspaces $1\in S_n=S_n^*\subset B$ satisfying $S_n^2\subseteq S_{n+1}$ for all $n\geq 1$ and such that the union $\bigcup_{n\geq 1} S_n$ is dense in $B$.  
It follows readily from these properties that the union $B_0= \bigcup_{n\geq 1} S_n$ is a dense unital $*$-subalgebra of $B$.

Let $A\subset C_\infty$ be a separable \cstar-subalgebra.
Then we can find a family of bounded sequences of self-adjoint elements $\{ (a_n^{(k)})_n \mid k\geq 1\}$ such that the resulting sequence $a^{(k)}:=[(a_n^{(k)})_n]$ forms a dense subset of $A_\mathrm{sa}$.
Given $m\geq 1$, we apply property \ref{P.approx-central-embeddings.3} to $\varepsilon=\frac1m$, $G_m:=\{ a^{(k)}_m \mid k\leq m\}$ in place of $F$, and $S_m$ in place of $S$ in order to choose a $*$-linear map $\kappa_m: S_m^2 \to C$ such that we have for all $c\in G_m$ and $b, b_1, b_2\in S_m$ that
		\begin{itemize}
		\item $\|\kappa_m(b)\|\leq \frac{m+1}{m} \|b\|$.
		\item $\| [\kappa_m(b),c]\|\leq \frac1m \|b\|$.
		\item $\|(\kappa_m(b_1)\kappa_m(b_2)-\kappa_m(b_1b_2))c\|\leq \frac1m\|b_1\| \|b_2\|$.
		\item $\|(1-\kappa_m(1))c\|\leq \frac1m$.
		\end{itemize}
We consider the $*$-preserving map $\phi: B_0\to\ell^\infty(C)$ given by
\[
\phi(b)_n = \begin{cases} 0 &,\quad b\notin S_n \\ \kappa_n(b) &,\quad b\in S_n. \end{cases}
\]
For all pairs $b_1, b_2\in B_0$, there exists some $m\geq 1$ with $b_1, b_2\in S_m$.
Since each map $\kappa_\ell$ is linear, this means that the induced map $\bar{\phi}: B_0\to C_\infty$ is a linear contraction.
We have for all $b\in B_0$ and $k\geq 1$ that
\[
\| [\bar{\phi}(b), a^{(k)}]\| = \limsup_{n\to\infty} \|[\kappa_n(b), a^{(k)}_n]\|=0,
\]
and
\[
\|(1-\bar{\phi}(1))a^{(k)}\|=\limsup_{n\to\infty} \|(1-\kappa_n(1))a^{(k)}\|=0,
\]
and
\[
\|(\bar{\phi}(b_1)\bar{\phi}(b_2)-\bar{\phi}(b_1b_2))a^{(k)}\| = \limsup_{n\to\infty} \|(\kappa_n(b_1)\kappa_n(b_2)-\kappa_n(b_1b_2))a^{(k)}_n\| = 0.
\]
This implies that $\bar{\phi}$ takes values in $C_\infty\cap A'$ and its induced map 
\[
\psi: B_0\to (C_\infty\cap A')(C_\infty\cap A^\perp)
\] 
is a unital $*$-homomorphism.
By continuity, it follows that $\psi$ extends uniquely to a unital $*$-homomorphism $\psi: B\to (C_\infty\cap A')/(C_\infty\cap A^\perp)$.
Since $A$ was chosen arbitrarily, this shows the claim.
\end{proof}

\begin{remark}
Readers fluent in model theory will recognize \ref{P.approx-central-embeddings.4}$\Rightarrow$\ref{P.approx-central-embeddings.1} in the proof above as an easy implication.
To wit, consider a countable dense $*$-$\bbQ+i\bbQ$-subalgebra $B_0$ of $B$.
Then \ref{P.approx-central-embeddings.4} asserts that the quantifier-free type of $B_0$ is relatively consistent with the theory of $(C_\infty \cap A')/(C\cap A^\perp)$.
By a proof similar to that of \cite[Proposition~16.5.3]{Fa:STCstar}, the latter \cstar-algebra is countably quantifier-free saturated and the type is realized in it,  giving a unital isometric embedding of $B_0$.
Its continuous extension to~$B$ is the required embedding.
This framework underpins the proof of the corollary below.
\end{remark}

The following can be viewed as a generalization of \cite[Theorem 2.5.2(21)]{Muenster}.

\begin{corollary} \label{cor:axiomatizable}
Suppose that $B$ is a separable unital \cstar-algebra.
	\begin{enumerate}[label=\textup{(\arabic*)}]
	\item \label{1.cor:axiomatizable} The property of admitting large approximately central maps from $B$ is  axiomatizable in the language of \cstar-algebras.
	\item \label{2.cor:axiomatizable} Suppose that $B=\calD$ is a strongly self-absorbing \cstar-algebra.
	Let $C$ be an arbitrary \cstar-algebra.
	Then $C$ admits large approximately central maps from $\calD$ if and only if $C$ is separably $\calD$-stable.
	\end{enumerate}
\end{corollary}
\begin{proof}
	\ref{1.cor:axiomatizable}:
	Fix a countable dense $*$-$\bbQ+i\bbQ$-subalgebra $\{b_n\mid n\in \bbN\}$ of $B$ enumerated so that $b_0=1$ and $b_{2n+1}=b_{2n}^*$ for all $n$.
	For every $m\geq 1$ consider the following  sentence ($\bar x$ and $\bar y$ stand for $2m$-tuples of variables where $x_j$ range over the unit ball and $y_j$ is of the same sort as $b_j$, i.e., it ranges over the $n_j$-ball where~$n_j$ is the minimal natural number not smaller than~$\|b_j\|$). 
	\begin{align*}
		\varphi_m:=\sup_{\|\bar x\|\leq 1}\inf_{\bar y} \max\biggl\{&\max_{i,j<2m} \|[x_i,y_j]\|,\max_{j<m}\|y_{2j}^*-y_{2j+1}\|, \biggr.\\
		& \max\{\|y_j y_k-y_l\|\mid j,k,l<2m, b_j b_k=b_l\},\\
		\biggl.&\max\{\|y_j-\alpha y_k\|\mid j,k<2m,\alpha\in \bbQ+i\bbQ, b_j=\alpha b_k\}\biggr\}. 
	\end{align*}
	Then $\varphi_m$ has value $0$ in $C$ if and only if for every $2m$-element subset $F$ of its unit ball and every $\varepsilon>0$ there is a self-adjoint $\varepsilon$-approximately linear map from the linear span $S_m$ of $\{b_j\mid j<2m\}$ into $C$ such that the images of $b_j$, for $j<2m$ $\varepsilon$-commute with all elements of $F$.\footnote{An observant reader will notice that, because the inf is not necessarily attained, we are getting only an approximately self-adjoint map. A moment of thought shows that it can be easily perturbed to a self-adjoint map at the expense of increasing the $\varepsilon$.}
	By compactness of the unit ball of $S_m$, the map can be assumed to be linear.  

	 By Proposition~\ref{P.approx-central-embeddings}\ref{P.approx-central-embeddings.4}, the existence of such maps is equivalent to the \cstar-algebra $C$ admitting approximately large central maps from $B$. 
	 Therefore this property is axiomatized by the conditions $\varphi_m=0$, for all $m\geq 1$. 
	 
	\ref{2.cor:axiomatizable}:
First assume that $C$ admits large approximately central maps from~$\calD$.
Then every separable \cstar-subalgebra of $C$ is included in a separable elementary submodel $A$ of $C$.
Since then $A$ also admits large approximately central maps from $\calD$ by part \ref{1.cor:axiomatizable}, it satisfies the conditions in Proposition~\ref{P.approx-central-embeddings}.
By Theorem~\ref{Thm.Cstar-McDuff}, $A$ is $\calD$-absorbing. 

Conversely, assume that $C$ is separably $\calD$-absorbing.
Since the inductive limit of a sequence of separable, $\calD$-absorbing \cstar-algebras is $\calD$-absorbing, this implies that the $\calD$-absorbing \cstar-algebras form a club in $\Sep(C)$.
The intersection of this club with the club of elementary submodels of $C$ is a club by Lemma~\ref{Lem.clubs}.
By Theorem~\ref{Thm.Cstar-McDuff} and considering that the conditions in Proposition~\ref{P.approx-central-embeddings} hold for all members of such a club, it follows that $C$ admits large approximately central maps from $\calD$. 
\end{proof}

\begin{definition}\label{Def.B-saturated}
	Suppose that $B$ and $C$ are unital \cstar-algebras such that~$B$ is separable.
	We say that $C$ is \emph{$B$-saturated} if for every separable \cstar-subalgebra $A$ of $C$ there is a unital $*$-embedding from $B$ to $C\cap A'$. 
\end{definition}

If the algebra $B$ above is in fact a strongly self-absorbing \cstar-algebra~$\calD$, then examples of $\calD$-saturated algebras are the ultrapower $A_\cU$ and the sequence algebra $A_\infty$ of a unital $\calD$-absorbing \cstar-algebra~$A$ (see \cite{effros1978c}, \cite[Theorem~7.2.2]{Ror:Classification},  and \cite{FaHaRoTi:Relative}), path algebras of unital $\calD$-saturated \cstar-algebras (\cite[Corollary~5.3]{farah2023calkin}) and coronas of $\sigma$-unital $\calD$-stable \cstar-algebras as per Theorem~\ref{T.Q-saturated}.

The requirement that the *-homomorphism be unital is important in the above definition.
For example, it was proven in \cite[Theorem~B]{farah2023calkin} that $\cQ(H)$ is not $\cO_\infty$-saturated while Voiculescu's theorem easily implies that for every separable \cstar-sub\-algebra $C$ of $\cQ(H)$, the relative commutant $\cQ(H)\cap C'$ contains an isomorphic (but non-unital) copy of $\cO_\infty$. 

If a separable \cstar-algebra $M$ is $\calD$-stable and $A$ is a hereditary \cstar-subalgebra of $M$, then $A$ is $\calD$-stable (this is \cite[Corollary~3.1]{ToWi:Strongly}).

\begin{proposition} \label{P.A-in-M(A)} 
	Suppose that $\calD$ is strongly self-absorbing.
	If $A$ is a hereditary \cstar-subalgebra of a separably $\calD$-stable \cstar-algebra $C$, then $A$ is separably $\calD$-stable. 
	In particular, if $A$ is a \cstar-algebra and $\cM(A)$ is separably $\calD$-stable, then so is $A$.  
\end{proposition}

\begin{proof} 
Suppose that $C$ is separably $\calD$-stable.
By Proposition~\ref{Prop.ssa-clubs}, the collection $\cC$ of all separable $\calD$-stable \cstar-subalgebras in $C$ is a club.
	
Then the collection 
\(
\{ A\cap B \mid B\in\cC\}
\) 
includes a 
a club in $\Sep(A)$ by Lemma~\ref{Lem.club.restriction}. 
As pointed out above, it follows that $A\cap B$ is $\calD$-stable for all $B\in\cC$, being a separable hereditary subalgebra of a $\calD$-stable \cstar-algebra.
Hence $A$ is separably $\calD$-absorbing. 
\end{proof}

\begin{proposition} \label{P.sep-D-is-stable} 
	Suppose that $A$ is a \cstar-algebra and $\calD$ is strongly self-absorbing.
	\begin{enumerate}[label=\textup{(\arabic*)}]
	\item If $A$ is separably $\calD$-stable, then it follows for every \cstar-algebra $B$ that $A\otimes_{\max} B$ is separably $\calD$-stable. \label{P.sep-D-is-stable:1}
	\item $A$ is separably $\calD$-stable if and only if $\cK\otimes A$ is separably $\calD$-stable. \label{P.sep-D-is-stable:2}
	\end{enumerate}
\end{proposition}

\begin{remark*}
Regarding statement \ref{P.sep-D-is-stable:1} above, notice that it also follows for any \cstar-algebra $B$ and any \cstar-crossnorm $\gamma$ on $A\odot B$ that $A\otimes_\gamma B$ is separably $\calD$-stable.
This is because such \cstar-algebras arise as quotients of $A\otimes_{\max} B$ and separable $\calD$-stability passes to quotients.
\end{remark*}

\begin{proof}[Proof of Proposition~\ref{P.sep-D-is-stable}]
\ref{P.sep-D-is-stable:1}:
For \cstar-subalgebras $B_0$ in $ \Sep(B)$ and $C$ in $\Sep(A)$, we denote (ad-hoc) by $C\otimes_{\bullet} B$ the image of the unique $*$-ho\-mo\-mor\-phism $C\otimes_{\max} B_0\to A\otimes_{\max} B$ given by $\pi(c\otimes b)=c\otimes b$ for all $c\in C$ and $b_0\in B$.\footnote{We caution the reader that for certain choices of the involved algebras this map is not necessarily an inclusion.}
Notice that if $\cC\subseteq\Sep(A)$ is any club, then the collection
\[
\{ C\otimes_{\bullet} B_0 \mid C\in\cC,\ B_0\in\Sep(B) \}
\]
is also a club in $A\otimes_{\max} B$.
So if $A$ is separably $\calD$-stable, then we can choose a club consisting of $\calD$-absorbing \cstar-algebras, leading to $A\otimes_{\max} B$ being separably $\calD$-stable.

\ref{P.sep-D-is-stable:2}:
This is a direct consequence of Proposition~\ref{P.A-in-M(A)} and part \ref{P.sep-D-is-stable:1} applied to $B=\cK$.
\end{proof}


\begin{lemma} \label{L.extensions}
Let $\calD$ be a strongly self-absorbing \cstar-algebra.
Suppose $A$ is a \cstar-algebra with a closed two-sided ideal $I\subseteq A$.
Then $A$ is separably $\calD$-stable if and only if $I$ and $A/I$ are separably $\calD$-stable.
\end{lemma}

\begin{proof}
Let $\pi\colon A\to A/I$ denote the quotient map.
Let us first briefly argue the ``only if'' part, so assume $A$ is separably $\calD$-stable.
As an ideal,~$I$ is also hereditary and thus  separably $\calD$-stable as a consequence of Proposition~\ref{P.A-in-M(A)}.
Furthermore, showing that $A/I$ is separably $\calD$-stable is a trivial exercise, given that quotients of separable $\calD$-stable \cstar-algebras are $\calD$-stable.
	
For the ``if'' part, assume both $I$ and $A/I$ are separably $\calD$-stable.
If $\cC\subseteq \Sep(A/I)$ is a club it is then not difficult to see that $\{B\in \Sep(A)\mid \pi(B)\in \cC\}$ is a club in $\Sep(A)$.
Similarly, if $\cC\subseteq \Sep(I)$ is a club, then $\{B\in \Sep(A)\mid B\cap I\in \cC\}$ is a club.
Let 
	\begin{align*}
		\cC_0=\{B\in \Sep(A)&\mid \text{  $\pi(B)$ is $\calD$-stable}\},\\
		\cC_1=\{B\in \Sep(A)& \mid  \text{ $B\cap I$ is $\calD$-stable}\}.
	\end{align*} 
By the assumptions and Proposition~\ref{Prop.ssa-clubs}, each one of $\cC_0$ and $\cC_1$ is a club, and $\cC_0\cap \cC_1$ is then a club by Lemma~\ref{Lem.clubs}.
	By \cite[Theorem~4.3]{ToWi:Strongly}, every $B\in \cC_0\cap \cC_1$ is $\calD$-absorbing, hence $A$ is separably $\calD$-stable by Proposition~\ref{Prop.ssa-clubs}. 
\end{proof}

\begin{corollary} \label{C.M(A)locally-tensorially-D-absorbing}
Let $A$ be any \cstar-algebra.
Assume that $\calD$ is strongly self-absorbing and that both $A$ and $\cQ(A)$ are separably $\calD$-stable.
Then $\cM(A)$ is separably $\calD$-stable. 
\end{corollary}

\section{On $\calD$-saturation of coronas} \label{S.D-saturation}

For proving the main result of this section we will need the existence of quasi-central approximate units, which were introduced in \cite{Arv:Notes}.
We first record a few basic facts about them.   
 If $E$ is an  ideal in a \cstar-algebra~$C$ and $X$ is a subset of~$C$, then an approximate unit $(e_\lambda)$ for $E$ is said to be \emph{$X$-quasicentral} if $\|[e_\lambda,x]\|\to 0$ for every $x\in X$ (see e.g., \cite[\S 1.9]{Fa:STCstar}). We are interested in the case when $E$ is $\sigma$-unital and the approximate unit is countable. In this situation, a sequential  $X$-quasicentral approximate unit can be chosen for every separable subset $X$. 
 We will need the following more precise statement.
 
\begin{lemma}\label{L.qc} Suppose that $E$ is a $\sigma$-unital  \cstar-algebra, $X$ is a countable subset of $C$, and $\e_n>0$, for $n\in \bbN$, are such that $\lim_n \e_n=0$. Then there is an approximate unit $e_n$, for $n\in \bbN$, of $E$ such that the following holds. 
	\begin{enumerate}[label=\textup{(\arabic*)}]
		\item $e_0=0$ and $e_{n+1} e_n=e_n$ for all $n$. 
		\item With 
		\[
		f_n=(e_{n+1}-e_n)^{1/2} 
		\]
		the series $\sum_n f_n^2$ strictly converges to $1_{\cM(E)}$. 
		\item  $f_m f_n=0$ whenever $|m-n|\geq 2$. 
		\item For every $x\in X$ and all large enough $n$ we have  $\|[f_n,x]\|<\e_n$. 
	\end{enumerate} 
\end{lemma}

\begin{proof} The assertion is trivial when $E$ is unital (take $e_n=1$ for all $n$), and we may assume $E$ is not unital.  Enumerate $X$ as $\{x_n\mid n\in \bbN\}$. 
Fix an approximate unit $e_n$, for $n\in \bbN$, in $E$ which is $E\cup \{x_n\mid n\in \bbN\}$-quasicentral and such that $e_{n+1}e_n=e_n$ for all $n$ (\cite[Proposition~1.9.3]{Fa:STCstar}). Note that this sequence does not necessarily satisfy the requirements. In order to ensure that all required conditions hold we will have to pass to a subsequence of this approximate unit, also denoted $(e_n)$.  We may assume that $e_0=0$.  By \cite[Lemma 1.4.8]{Fa:STCstar} for each $n$ there is $\delta_n>0$  such that for contractions $a\geq 0$ and $x$ we have that $\|[a,x]\|<\delta_n$ implies $\|[a^{1/2},x]\|<\e_n$. By going to a subsequence we may assume that $\|[e_n,x_j]\|<\delta_n/2$ for all $j\leq n$. Let, for $n\geq 0$, 
\[
f_n=(e_{n+1}-e_n)^{1/2}. 
\]
Clearly the series $\sum_n f_n^2$ strictly converges to $1_{\cM(E)}$. 
Also, if $|m-n|\geq 2$ then $f_m^2 f_n^2=0$, and by applying continuous functional calculus  one obtains $f_m f_n=0$. Also $\|[f_n,x_j]\|<\varepsilon_n$ for all $j\leq n$, as required. 
\end{proof}

\begin{lemma}\label{L.fn}
	Suppose that $f_n$, for $n\in \bbN$, is a sequence of positive contractions in a multiplier algebra $\cM(E)$ such that the series $\sum_n f_n^2$ strictly converges to $1_{\cM(E)}$ and one has $f_nf_m=0$ whenever $|n-m|\geq 2$. 
	
	Then for every $Y\subseteq \bbN$ the function $\Psi_Y\colon \ell_\infty(\cM(E))\to \cM(E)$ defined by 
	\[
	\Psi_Y((a_n))=\sum_{n\in Y} f_n a_n f_n
	\]
	is completely positive and contractive. 
\end{lemma}

\begin{proof}
The sum in the definition of $\Psi_Y$  is strictly convergent (\cite[Claim 1 on p. 370]{Fa:STCstar}), thus $\Psi_Y$ is well-defined. It is completely positive because each of the summands is completely positive (this is similar to \cite[Claim 2 on p. 370]{Fa:STCstar}). Since $\Psi_Y(1)=\sum_{n\in Y} f_n^2\leq 1_{\cM(E)}$, $\Psi_Y$ is contractive.  
\end{proof}

The following is \cite[Lemma 3.1]{manuilov2004theory} but we include a proof for the reader's convenience ($\pi\colon\cM(E)\to \cQ(E)$ denotes the quotient map).

\begin{lemma}\label{L.MTestimate}
	Suppose that $E$ is $\sigma$-unital and $f_n$, for $n\in \bbN$, is a sequence of positive contractions in $E$ such that the sum $\sum_n f_n^2$ strictly converges to~$1_{\cM(E)}$ and $f_m f_n=0$ if $|m-n|\geq 2$.
	Then for every $k\in \bbZ$ and every sequence $(a_n)$ in $\ell_\infty(\cM(E))$ the sum $\sum_n f_n a_n f_{n+k}$ strictly converges to an element of $\cM(E)$ and 
		$\|\sum_n f_n a_n f_{n+k}\|\leq \sup_n \|a_n\|$. 
			We  also have   
	\[
	\textstyle	\|\pi(\sum_n f_n a_n f_{n+k})\|\leq \limsup_n \|a_n\|. 
	\]	
\end{lemma}

\begin{proof}
Fix natural numbers $i<j$, and consider $\cM(E)^{\oplus(j-i)}$ as a Hilbert $\cM(E)$-module.
Fix a contraction $b\in \cM(E)$.
By the generalized Cauchy--Schwarz inequality we have  the following. 
	\begin{align*}
	\textstyle	\|\sum_{n=i}^j f_n a_n f_{n+k}\|^2&\textstyle\leq \|\sum_{n=i}^j f_n a_n a_n^* f_n\|
		\|\sum_{n=i}^j  f_{n+k}^2 b \|\\
		 &\textstyle\leq \sup_{n\geq i}\|a_n\|^2 \|\sum_{n=i}^j f_n^2 \| \|\sum_{n=i}^j b^* f_{n+k}^2 b\| \\
		 &\textstyle\leq \sup_{n\geq i} \|a_n\|^2.
	\end{align*}
	With respect to the strict topology, we have $\lim_{j\to \infty} \sum_{n=i}^j f_n^2=(1-e_i)$ and $\lim_{j\to \infty}\sum_{n=i}^j b^* f_{n+k}^2 b=b^*(1-e_{i+k})b$.
	Since $\lim_{i\to \infty} \|b^*(1-e_{i+k})b\|=0$ when $b\in E$, the series is strictly convergent.  
	
	When inserting $b=1$ and letting $j\to\infty$, the above estimate shows $\|\sum_{n\geq i} f_n a_n f_{n+k}\| \leq \sup_{n\geq i}\|a_n\|$ for all $i$.
	This implies both of the desired norm estimates.   
\end{proof}


\begin{lemma} \label{Lem.separable-exhaustion-multipliers}
Let $A$ be a $\sigma$-unital \cstar-algebra.
Given any strictly separable subset $X$ of $\cM(A)$, there exists a nondegenerate separable \cstar-subalgebra $B\subseteq A$ such that under the canonical inclusion $\cM(B)\subseteq\cM(A)$, one has $X\subseteq\cM(B)$.

As a consequence, if $\cC$ is any cofinal subset of $\Sep(A)$, then for any separable \cstar-subalgebra $C$ of $\cQ(A)$ there exists $B\in\cC$ with $C\subseteq\cQ(B)$.
\end{lemma}
\begin{proof}
The ``as a consequence'' part follows from functoriality of the multiplier algebra under nondegenerate inclusions of \cstar-algebras. First note that since $A$ is $\sigma$-unital, the set of $B\in \cC$ that include an approximate unit for~$A$ (and are therefore nondegenerate) is cofinal in $\cC$. 
Given a nondegenerate inclusion $B\subseteq A$, if $\cM(B)\subseteq\cM(A)$ is the induced inclusion, then the homomorphism theorem also gives us an inclusion $\cQ(B)\subseteq\cQ(A)$.
If the first part of the statement holds and we pick a separable \cstar-subalgebra $C$ of $\cQ(A)$, then we may choose a separable subset $X\subseteq\cM(A)$ whose image under the quotient map is $C$, and apply the first part of the statement to find the desired element $B\in\cC$.

Hence it suffices to prove the first part of the statement.
We may assume without loss of generality that $X$ is self-adjoint and contains the unit.
We pick a countable approximate unit $e_n\in A$.
Let $B$ be the \cstar-algebra generated by $\{ xe_n\mid n\in\mathbb N,\ x\in X\}$, which is norm-separable if $X$ is strictly separable.
Evidently $B$ is nondegenerate in $A$ because it contains all $e_n$.
Furthermore, one has for all $x,y\in X$ that
\[
x(ye_n)=\lim_{m\to\infty} xe_m ye_n \in B.
\]
In this way it follows that $xB\subseteq B$ for all $x\in X$.
Since $X$ was self-adjoint, we have also $Bx\subseteq B$ for all $x\in X$.
Hence every $x\in X$ is automatically a multiplier on $B$ and the claim is proved.
\end{proof}

We can now prove Theorem~\ref{T.Q-saturated}. 

\begin{theorem}\label{T.Q-saturated.1}
Suppose that $\calD$ is a strongly self-absorbing \cstar-algebra and that $A$ is a $\sigma$-unital \cstar-algebra that is separably $\calD$-stable.
Then $\cQ(A)$ is $\calD$-saturated. 
\end{theorem}
\begin{proof}
Let $h\in A$ be a strictly positive element.
As $A$ is separably $\calD$-stable, the collection $\cC\subseteq\Sep(A)$ of separable $\calD$-stable \cstar-subalgebras in $A$ containing $h$ is a club.
Such subalgebras are automatically nondegenerate.
By Lemma~\ref{Lem.separable-exhaustion-multipliers}, we see that every separable \cstar-subalgebra of $\cQ(A)$ can be seen as a subalgebra of $\cQ(B)$ for some $B\in\cC$.
Considering the definition of $\calD$-saturation, it suffices to show that $\cQ(B)$ is $\calD$-saturated for every $B\in\cC$.
To put it differently, it suffices to prove the conclusion of the theorem in case when  $A$ is separable.
Let us assume that as of now.

Suppose that $B$ is any separable \cstar-subalgebra of $\cQ(A)$.
By \cite[Lem\-ma~6.8]{hirshberg2017rokhlin}, there exists a unital $*$-homomorphism from $\calD$ into $\cQ(A)\cap B'$ if and only if there exist two completely positive maps of order zero $\tilde\Psi_0, \tilde\Psi_1:\calD\to\cQ(A)\cap B'$  whose ranges commute pointwise and such that $\tilde\Psi_0(1)+\tilde\Psi_1(1)=1$. 
We proceed to find such maps. 
	
Fix a countable subset $\{b_n\mid n\in \bbN\}$ of the unit ball of $\cM(A)$ whose image under the quotient map is dense in the unit ball of $B$. 
By Lemma~\ref{L.qc} there is an approximate unit $e_n$, for $n\in \bbN$, of $A$ such that with 
\[
f_n=(e_{n+1}-e_n)^{1/2} 
\]
the series $\sum_n f_n^2$ strictly converges to $1_{\cM(A)}$, $f_m f_n=0$ whenever $|m-n|\geq 2$, and furthermore we have  
\begin{equation} \label{eq:f-and-b}
\|[f_n,b_j]\|<2^{-n} \quad\text{for all } 1\leq j\leq n. 
\end{equation}
 Fix an increasing sequence $F_n$, for $n\in \bbN$, of finite subsets of the unit ball of $\calD$ with dense union. 
 Since we assumed $A$ to be separable, we have $A\cong A\otimes \calD$.
 As $\calD$ embeds approximately centrally into itself, we can find a sequence of unital $*$-homomorphisms $\pi_n\colon \calD\to \cM(A)$, for $n\in \bbN$, such that $[\pi_n(d),a]\to 0$ in norm for all $a\in A$ and $d\in\calD$.
By passing to a subsequence, if necessary, we may ensure that the sequence $\pi_n$ satisfies the following conditions for all $n$:
 \begin{enumerate}[label=(\roman*)]
  	\item $\|[\pi_n(d),f_m]\|<2^{-n}$ for all $d\in F_n$ and $m\leq n+1$, 
 	\item $\|[\pi_n(d),f_n b_k]\|<2^{-n}$ for all $d\in F_n$ and $k<n$,
	\item $\|[\pi_n(d_1), f_j \pi_j(d_2) f_j]\|<2^{-n}$ for all $d_1,d_2\in F_n$ and $j<n$. 
 \end{enumerate}
Define $\Psi_i\colon \calD\to \cM(A)$ for $i=0,1$ by 
\[
\Psi_i(d):=\sum_{n=0}^\infty f_{2n+i}\pi_{2n+i}(d)f_{2n+i}. 
\]
Each $\Psi_i$ is well-defined and completely positive by Lemma~\ref{L.fn}. 
Clearly $\Psi_0(1)+\Psi_1(1)=1$.
For $i=0,1$, let $\tilde\Psi_i: \calD\to\cQ(A)$ be the composition with the quotient map $\cM(A)\to\cQ(A)$.
Then $\tilde \Psi_0$ and $\tilde\Psi_1$ are completely positive maps from $\calD$ into $\cQ(A)$ such that $\tilde\Psi_0(1)+\tilde\Psi_1(1)=1$. 

First we observe for all $n$ and $k<n$ that 
\begin{equation*}
	f_n \pi_n(d) f_n b_k \approx_{2^{-n}} 
		f_n^2 b_k \pi_n(d) 
		 \approx_{2^{-n+1}} b_k f_n^2 \pi_n(d)
		 \approx_{2^{-n}} b_k f_n \pi_n(d) f_n.
\end{equation*}
Therefore $\|[f_n \pi_n(d) f_n ,b_k]\|<2^{-n+2}$ and   $[\Psi_i(d),b_k]\in A$ for $i=0,1$, all $d\in \bigcup_n F_n$, and all $k$.
By continuity, the range of the induced map $\tilde\Psi_i: \calD\to\cQ(A)$ is included in $\cQ(A)\cap B'$. 

Our next task is to prove that the ranges of $\Psi_0$ and $\Psi_1$ commute pointwise modulo $A$.
For all $n$ and $d\in F_n$ we have
\[
f_{n+1}^2 f_n \pi_n(d) f_n\approx_{2^{-n+1}} 
f_n \pi_n(d) f_n  f_{n+1}^2.
\]
Finally, for all $n$ and $d_1, d_2\in F_n$ we get  
\begin{align*}
f_{n+1} \pi_{n+1}(d_1) f_{n+1} f_n \pi_n(d_2) f_n
&\approx_{2^{-n}} 
f_{n+1}^2 \pi_{n+1}(d_1) f_n \pi_n(d_2)f_n\\
&\approx_{2^{-n}} 
f_{n+1}^2 f_n \pi_n(d_2) f_n \pi_{n+1}(d_1)\\
&\approx_{2^{-n+1}}
 f_n \pi_n(d_2) f_n f_{n+1}^2 \pi_{n+1}(d_1)\\
&\approx_{2^{-n}}
f_n \pi_n(d_2) f_n f_{n+1} \pi_{n+1}(d_1)f_{n+1}.
\end{align*}
Since $f_m f_n=0$ if $|m-n|\geq 2$, this shows that for all $m>n$, $d\in F_n$, and $d'\in F_m$ we have $\|[f_m \pi_m(d') f_m, f_n \pi_n(d) f_n]\|<5\cdot 2^{-n}$.
Since $\bigcup_n F_n$ is dense in the unit ball of $\calD$, we get $[\Psi_0(d_1),\Psi_1(d_2)]\in A$  for all $d_1,d_2\in\calD$. 
Equivalently, we get that the ranges of $\tilde\Psi_0$ and $\tilde\Psi_1$ commute pointwise.

We lastly verify that for $i=0,1$, $\tilde\Psi_i$ is of order zero. 
Let $k\in\mathbb N$ and $a,b\in F_k$ be two elements.
Because $f_m f_n=0$ if $|m-n|\geq 2$, we have 
\[
\Psi_i(a)\Psi_i(b)=\sum_{n=0}^\infty f_{2n+i} \pi_{2n+i}(a)f_{2n+i}^2 \pi_{2n+i}(b) f_{2n+i}.
\] 
Since $\|[f_{n}, \pi_n(a)]\|<2^{-n+1}$ for $n+1\geq k$, 
and each $\pi_n$ is a unital $*$-homomorphism, it follows that
\[
\Psi_i(a)\Psi_i(b)-\Psi(1)\Psi_i(ab) =\sum_{n=0}^\infty f_{2n+i} [\pi_{2n+i}(a),f_{2n+i}^2] \pi_{2n+i}(b) f_{2n+i} \in A.
\]  
By continuity, we obtain the equality $\tilde\Psi_i(a)\tilde\Psi_i(b)=\tilde\Psi_i(1)\tilde\Psi_i(ab)$ for all $a,b$ in $\calD$.
In conclusion, $\tilde\Psi_i$ is of order zero. 

By what we remarked earlier (using \cite[Lemma~6.8]{hirshberg2017rokhlin}), it follows that there exists a unital $*$-homomorphism from $\calD$ to $\cQ(A)\cap B'$. 
Since $B$ was chosen arbitrarily, the proof is complete.
\end{proof}

Here is a partial converse to Proposition~\ref{P.A-in-M(A)}, which treats a consequence of (separable) $\calD$-stability that has not been previously observed in the literature to the best of our knowledge. 

\begin{lemma}	\label{L.A-in-M(A)}
Suppose that $\calD$ is strongly self-absorbing and $A$ is a $\sigma$-unital \cstar-algebra. 
	Then $A$ is separably $\calD$-stable if and only if $\cM(A)$ is separably $\calD$-stable.
\end{lemma}

\begin{proof}
The ``if'' part is settled by Proposition~\ref{P.A-in-M(A)}.
For the ``only if'' part, assume $A$ is separably $\calD$-stable.
We give two different short proofs for this implication using Theorem~\ref{T.Q-saturated.1}.

Firstly, by Theorem~\ref{T.Q-saturated.1} we may conclude that $\cQ(A)$ is $\calD$-saturated, so in particular also separably $\calD$-stable.
By Corollary~\ref{C.M(A)locally-tensorially-D-absorbing}, it follows that $\cM(A)$ is separably $\calD$-stable as well. 
Alternatively, we know by Proposition~\ref{P.sep-D-is-stable}\ref{P.sep-D-is-stable:1} that $c_0(A)\cong c_0(\bbN)\otimes A$ is separably $\calD$-stable, hence by Theorem~\ref{T.Q-saturated.1} it follows that its corona $\cQ(c_0(A))=\ell^\infty(\cM(A))/c_0(A)$ is $\calD$-saturated.
The sequence algebra $\cM(A)_\infty=\ell^\infty(\cM(A))/c_0(\cM(A))$ is a quotient of this \cstar-algebra and hence also $\calD$-saturated, which is equivalent to $\cM(A)$ being separably $\calD$-stable via Proposition~\ref{P.approx-central-embeddings}.
\end{proof}

The conclusion of Theorem~\ref{T.Q-saturated.1} may fail if $A$ is not assumed to be $\sigma$-unital. 

\begin{example} \label{ex:1-dim-corona}
Let $\bbU$ be the universal UHF algebra, given as the infinite tensor product $\bigoplus_{n\geq 2} M_n(\bbC)$, which is strongly self-absorbing.
There is a separably $\bbU$-stable \cstar-algebra $A$ such that $\cQ(A)$ is one-dimensional and hence not even separably $\cZ$-stable.
Choose a nonprincipal ultrafilter $\cU$ on $\bbN$ and consider $\calR^\cU$, the (tracial) ultrapower of the hyperfinite II$_1$-factor $\calR$.
Since $\calR$ is McDuff, it follows for any von Neumann subalgebra $B$ of $\calR^\cU$ with separable predual that $\calR^\cU\cap B'$ includes a unital copy of $\bbU$ (and even of $\calR$).
This means that $\calR^\cU$ is $\bbU$-saturated, so in particular separably $\bbU$-stable as a \cstar-algebra.
Let $L$ be a maximal left ideal of $\calR^\cU$ and let $A=L\cap L^*$.
This is a hereditary \cstar-subalgebra of $\calR^\cU$, and it is therefore separably $\bbU$-stable by Proposition~\ref{P.A-in-M(A)}.
By \cite[Theorem~1]{sakai1971derivations}, it follows that $\cQ(A)$ is one-dimensional. 
\end{example}

In relation to this example, we ought to mention that it is not known whether $\calR$ itself is (as a \cstar-algebra) separably $\cZ$-stable.
Our results do not shed light on this problem, however.

\section{A refinement of Schur’s Lemma for Property (T) groups}
\label{S.Schur} 

The main results of the present section are Lemma~\ref{L.Schur} and  Lemma~\ref{L.Schur.spicy.version} and they will serve as the technical backbone for the next section. 

Let $A$ be a \cstar-algebra.
By $\cS(A)$ we denote the space of states on $A$.
In the proofs of the lemmas below we use the theory of Hilbert modules. In particular we will need the the fact that $\cM(A)=\calL_A(A)$ is nothing but the algebra of all adjointable $A$-linear operators on $A$ when we view it as a Hilbert module over itself in a trivial way (\cite[\S II.7.3.11]{Black:Operator}).
For $m\geq 1$, we denote by $\tau_m$ the tracial state on the scalar $m\times m$ matrices, and define the faithful conditional expectation 
\begin{equation}\label{eq.Em}
\cE_m: M_m(A)=M_m\otimes A\to A \quad\text{via}\quad \cE_m=\tau_m\otimes\id_A.
\end{equation}
Note that $\cE_m$ extends to a conditional expectation $M_m(\mathcal M(A))\to \mathcal M(A)$ that is strictly continuous on norm-bounded sets.
We make use of the canonical identification $M_m(\mathcal M(A))=\mathcal L_A( A^{\oplus m} )$, where the latter is the set of all adjointable right $A$-linear endomorphisms on the right $A$-Hilbert module~$A^{\oplus m}$ (see e.g., \cite[\S II.7.1]{Black:Operator}).

\begin{notation}
Let $A$ be a \cstar-algebra and $m\geq 2$.
Given any right $A$-Hilbert module $\mathcal Y$, we can put a norm on $\mathcal L_A(A^{\oplus m},\mathcal Y)$ via 
\[
\|X\|_{m,A}=\|\cE_m(X^*X)\|^{1/2}.
\]
(Note that for $\mathcal Y=A^{\oplus m}$, we put this norm on $M_m(A)$.)
\end{notation}
We observe (as a consequence of the definition) that this norm obeys the equality
\begin{equation} \label{eq:generalized-HS-equality}
	\|1_m\otimes a\|_{m,A}=\|a\|
\end{equation}
and 
for all $a\in\mathcal M(A)$, $X\in \mathcal L_A(A^{\oplus m},\mathcal Y)$,  and $Y\in M_m(\mathcal M(A))$ the inequalities
\begin{equation} \label{eq:generalized-HS-inequalities}
	\|X(1\otimes a)\|_{m,A}\leq\|a\|\|X\|_{m,A} \quad\text{and}\quad \|XY\|_{m,A}\leq\|X\|\|Y\|_{m,A}.
\end{equation}
Here is another well-known fact included for reference.

\begin{lemma}\label{L.state-in-M(A)}
Suppose that $A$ is a \cstar-algebra.
\begin{enumerate}[label=\textup{(\arabic*)}]
	\item Every state $\varphi$ of $A$ has a unique extension to a state $\tilde\varphi$ of $\cM(A)$. 
	\item For every $T\in \cM(A)$ we have $\|T\|=\sup\{\tilde \varphi(T^*T)^{1/2}\mid \varphi\in \cS(A)\}$.
\end{enumerate}
\end{lemma}

\begin{proof}
The first part follows e.g., from \cite[Lemma~1.8.5]{Fa:STCstar}. For the second part, fix an approximate unit $(e_\lambda)$ in $A$.
Then $T$ is a strict limit of $e_\lambda^{1/2} T$ and $\|T\|=\lim_\lambda \|e_\lambda^{1/2} T\|$.
Since for every $a$ in $A$ there is a state of $A$ such that $\varphi(a^*a)=\|a\|^2$, for every $\lambda$ there is $\varphi_\lambda\in \cS(A)$ such that 
	\[
	\|T^*e_\lambda T\|=\varphi_\lambda(T^*e_\lambda T)\leq \varphi_\lambda(T^*T),
	\] 
	and the conclusion follows. 
\end{proof}

\begin{remark*}
Although it is irrelevant to what follows below, we note that in the second part of Lemma~\ref{L.state-in-M(A)} the supremum is typically not attained.
Choose $A$ to be an arbitrary non-unital but $\sigma$-unital \cstar-algebra so that some stictly positive contraction $e\in A$ exists.
Then $T=(1-e)^{1/2}\in\cM(A)$ has norm one, but $\tilde\varphi(T^2)<1$ for every $\varphi\in \cS(A)$ since $e$ is strictly positive. 
\end{remark*}

For the reader's convenience, in the following discussion we reiterate some of the points from \cite[\S 1.3]{farah2023calkin}.
Recall that a discrete group $\Gamma$ has \emph{Kazhdan's property (T)} if there are $F\Subset \Gamma$ and $\varepsilon>0$ such that for every unitary representation $\rho$ of $\Gamma$ on a Hilbert space~$H$, if there is a vector $\xi$ in~$H$ such that
\begin{equation} \label{Eq.Inv}
	\max_{g\in F} \|\rho(g)\xi-\xi\|<\varepsilon\|\xi\|.
\end{equation}
for all $g\in F$, (such $\xi\in H$  is called \emph{$(F,\varepsilon)$-invariant}), then there is a unit vector in $H$ that is invariant for $\rho$.
The pair $(F,\varepsilon)$ is called a Kazhdan pair for $\Gamma$. 
Suppose $\rho$ is an irreducible representation of $\Gamma$.
Then $\rho$ extends to a representation of the full group \cstar-algebra $\cstar(\Gamma)$ (see \cite{BrOz:C*} for information on group \cstar-algebras). The central cover of $\rho$ (\cite[Definition 1.4.2]{BrOz:C*}) is a Kazhdan projection in the second dual, $\cstar(\Gamma)^{**}$.
However, if $\Gamma$ has Kazhdan's property (T) then~$p_\rho$ belongs to $\cstar(\Gamma)$ 
(\cite[Definition 17.2.3 and Theorem 17.2.4]{BrOz:C*}). Kazhdan projections associated to inequivalent irreducible representations are orthogonal.
See \cite{bekka2008kazhdan}, \cite{BrOz:C*} for property (T) groups and \cite[\S 17]{BrOz:C*}, \cite[3.7.6]{higson2000analytic} for Kazhdan projections. 

Suppose that $\rho$ is an action of a discrete group $\Gamma$ on a Hilbert space $H$. 
Following \cite{farah2023calkin}, we say that  $\xi\in H$ is \emph{scaled $(F,\varepsilon)$-invariant} if
\begin{equation} \label{Eq.Scaled} 
	\max_{g\in F}\|\rho(g)\xi-\xi\|<\varepsilon.
\end{equation} 
Clearly every $\xi$ is scaled $(F,\varepsilon)$-invariant for any $F$ and $\varepsilon>2\|\xi\|$,  but in this case the estimate given in Lemma~\ref{L.bekka} below is vacuous. 

Let $q^\rho$ denote the orthogonal projection to the space of invariant vectors, 
\[
H^\rho:=\{\eta\in H\mid \rho(g)\eta=\eta\text{ for all } g\in \Gamma\}. 
\]
The following is \cite[Proposition 12.1.6]{BrOz:C*} (with $\Gamma=\Lambda$) or  \cite[Proposition~1.1.9]{bekka2008kazhdan} (the case when $F$ is compact, and modulo rescaling). 

\begin{lemma} \label{L.bekka} 
	Suppose  $\Gamma$ is a group with Kazhdan's property (T) and $(F,\varepsilon)$ is a Kazhdan pair for $\Gamma$.
	If $\Gamma$ acts on a Hilbert space $H$ then every scaled~$(F,\delta\varepsilon)$-invariant vector $\xi$ satisfies $\|\xi-q^\rho \xi\|<\delta$.
\end{lemma}

Suppose that $K$ and $H$ are Hilbert spaces. 
If $K$ is finite-dimensional let $\tau_K$ denote the normalized tracial state on $\cB(K)$ and 
consider the normalized Hilbert--Schmidt norm on $\cB(K,H)$, defined by
\[
\kHS T:=\tau_H(TT^*)^{1/2}.  
\]
With this norm, we note that $\cB(K,H)$ becomes a Hilbert space. 
The lemma below is a strengthening of \cite[Lemma~1.2]{farah2023calkin} which is in turn a minor variation on Schur's Lemma for property~(T) groups (\cite[Lemma~3.7.8]{higson2000analytic}).

\begin{lemma} \label{L.Schur} 
	Suppose $\Gamma$ is a property (T) group with Kazhdan pair $(F,\varepsilon)$ and $\rho_j$ is a  unitary representation of $\Gamma$ on a finite-dimensional Hilbert space~$K_j$ for $j=1,2$.
	Define a unitary representation $\sigma=\sigma_{\rho_1,\rho_2}$ of $\Gamma$ on $\cB(K_1,K_2)$ (viewed as a Hilbert space via the Hilbert--Schmidt norm defined above) via 
	\[
	\sigma_g(T):=\rho_2(g) T \rho_1 (g)^*,\quad T\in \cB(K_1,K_2). 
	\]
	Let $H$ be another Hilbert space and let $T \in \cB(K_1,K_2)\hat{\otimes} H$ satisfy 
	\[
	\max_{g\in F} \|(\sigma_g\otimes 1_H)(T)-T\|<\varepsilon\delta.
	\]
	\begin{enumerate}[label=\textup{(\arabic*)}]
		\item \label{1.Schur} If $K_1=K_2$ and $\rho_1=\rho_2$ is irreducible, then  $\|T-(\tau_1\otimes 1_H)(T)\|<\delta$, where $\tau_1$ is the normalized tracial state on $\cB(K_1)$.\footnote{Note that $\|\cdot\|$ denotes the product Hilbert space norm and we view $\tau_1\otimes 1_H$ as the orthogonal projection onto $\mathbb{C} 1_{K_1}\otimes H$ that is given on elementary tensors via $a\otimes\xi\mapsto\tau_1(a)\otimes\xi$.}
		\item \label{2.Schur} If $\rho_1$, $\rho_2$ have no isomorphic subrepresentations, then $\|T\|<\delta$. 
	\end{enumerate}
\end{lemma}

\begin{proof} Since $\oneHS\cdot\leq \|\cdot\|$, we have
	\begin{equation}\label{Eq.HS-inequality}
		\max_{g\in F} \oneHS{T\rho_1(g)-\rho_2(g)T}<\varepsilon\delta
	\end{equation}
	Consider the space $\cB(K_1,K_2)^\sigma$ of vectors invariant for $\sigma$, and let $Q$ be the projection to this space.
	Note that we have 
	\[
	\big(\cB(K_1,K_2)\hat{\otimes} H\big)^{\sigma\otimes 1_H}=\cB(K_1,K_2)^\sigma\hat{\otimes} H
	\] 
	with orthogonal projection $Q\otimes 1_H$.
	Since $(F,\varepsilon)$ is a Kazhdan pair for $\Gamma$, by~\eqref{Eq.HS-inequality} and  Lemma~\ref{L.bekka} we have 
	\begin{equation}\label{Eq.Q}
		\|T-(Q\otimes 1_H)(T)\|<\delta. 
	\end{equation}
	Now consider the two specific cases from the statement of the lemma. 
	
\ref{1.Schur}:
If $\rho_1=\rho_2$ is irreducible, then (identifying scalars with  scalar matrices), Schur's Lemma (\cite[Lemma~3.7.7]{higson2000analytic}) implies $\cB(K_1,K_1)^\sigma=\bbC 1_{K_1}$ and $Q(T)=\tau_1(T)$. Therefore \eqref{Eq.Q} reduces to $\|T-(\tau_1\otimes 1_H)(T)\|<\delta$.  
	
\ref{2.Schur}: 
We claim $\cB(K_1,K_2)^\sigma=\{0\}$ or equivalently $Q=0$.
	Since $K_1$ is finite-dimensional, $\rho_1$ is a direct sum of irreducible representations, say $(\rho_1,K_1)=\bigoplus_{\ell=1}^k (\rho^\ell_1, K_1^\ell)$.
	Then one gets a $\sigma$-equivariant direct sum decomposition $\cB(K_1,K_2)\cong\bigoplus_{\ell=1}^k \cB(K_1^\ell,K_2)$, which in particular implies $\cB(K_1,K_2)^\sigma \cong \bigoplus_{\ell=1}^k \cB(K_1^\ell,K_2)^{\sigma}$.
	Then Schur's Lemma (\cite[Lemma~3.7.7]{higson2000analytic}) implies that there are no nontrivial intertwiners for $\rho_1^\ell$ and $\rho_2$, $\ell=1,\dots,k$, which leads to $\cB(K_1,K_2)^\sigma=\{0\}$.
	Hence $Q(T)=0$ and \eqref{Eq.Q} reduces to  $\|T\|<\delta$. 
\end{proof} 

If $A$ is a \cstar-algebra, $H$ is a Hilbert space, and $\cK$ is the ideal of compact operators on $\cB(H)$, then $\cB(H)\cong \cM(\cK)$ can be identified with a subalgebra of $\cM(\cK\otimes A)$ for any \cstar-algebra $A$ via
\begin{equation}\label{eq.B(H)-inside-M(KotimesA)}
	\cM(\cK)\otimes 1_{\cM(A)}\subseteq \cM(\cK\otimes A). 
\end{equation}
The special case $A=\bbC$ within the following lemma is included in the proof of  \cite[Proposition~2.4]{farah2023calkin}. 

\begin{lemma} \label{L.Schur.spicy.version} 
	Suppose $\Gamma$ is a property (T) group with Kazhdan pair $(F,\varepsilon)$.
	Let $\rho_n: \Gamma\to\cU(H_n)$ be a sequence of pairwise inequivalent irreducible unitary representations with $m_n=\dim(H_n)<\infty$.
	Let $\rho$ be the direct sum representation on $H=\bigoplus_{n=1}^\infty H_n$ and let $p_n$ be the orthogonal projection onto $H_n$ for each $n\geq 1$.
	Let $A$ be a \cstar-algebra.
	As in \eqref{eq.B(H)-inside-M(KotimesA)}, we identify $\mathcal B(H)$ with a subalgebra of $\mathcal M(\cK\otimes A)$ and, with a slight abuse of notation, 
	identify $p_n$ with $p_n\otimes 1_{\cM(A)}$.  %
	Suppose that $n\geq 1$ is a natural number and a multiplier $T\in\mathcal M(\cK\otimes A)$ satisfies
	\[
	\max_{g\in F} \| \big( \rho(g)T\rho(g)^*-T \big)p_n\|<\varepsilon\delta.
	\]
	Then we have $\|Tp_n-p_n\otimes \cE_{m_n}(p_nTp_n)\|_{m_n,A}\leq 2\delta$.
\end{lemma}

\begin{proof}
	Given that $\cE_{m_n}$ is strictly continuous on bounded sets and $P_N:=\sum_{j=1}^N p_j$ is an increasing sequence that strictly converges to the unit,  for any $x\in\mathcal M(\cK\otimes A)p_n$ we have that $\cE_{m_n}(x^*P_Nx)$ is increasing and converges to $\cE_{m_n}(x^*x)$ strictly.
	Thus
	\[
	\lim_{N\to\infty} \|P_Nx\|_{m_n,A} = \lim_{N\to \infty}\| \cE_{m_n}(x^*P_N x)\|^{1/2} = \|x\|_{m_n,A},\quad x\in\mathcal M(\cK\otimes A)p_n.
	\]
	Applying this to $x=Tp_n$ and given that the range of $\rho$ pointwise commutes with all the projections $p_j$, it suffices to show the claim for $N\geq n$ and $T\in P_N\mathcal M(\cK\otimes A)p_n$.
	So let us assume that this is the case, and let us fix the numbers $N$ and $n$ for the rest of the proof.
	We may identify $\mathcal M(\cK\otimes A)$ with $\mathcal L_A( H\otimes_{\mathbb C} A)$, where $H\otimes_{\mathbb C} A$ is to be understood as the internal tensor product (see \cite[\S II.7.4]{Black:Operator}).
	Set $K=\bigoplus_{k=1}^N H_k$.
	Because the scalar projections~$p_n$ and $P_N$ have finite rank,  there is a bijective linear map
	\[
	\Phi: \cB( H_n,K)\otimes \mathcal M(A) \to \mathcal L_A( H_n\otimes_{\mathbb C} A, K \otimes_{\mathbb C} A) = P_N\mathcal M(\cK\otimes A)p_n
	\]
given on elementary tensors  by
	\[
	\Phi(Z\otimes x)(\xi\otimes a) = Z\xi\otimes xa.
	\]
	for all $Z\in\cB(H_n,K)$, $x\in\mathcal M(A)$, $\xi\in H_n$ and $a\in A$.
	One easily sees from this formula that $\Phi$ restricts further to a bijection between $\cB(H_n,H_n)\otimes \mathcal M(A)$ and $p_n\mathcal M(\cK\otimes A)p_n$. 
	
	Given $\phi\in S(A)$, define the semi-norm $\|\cdot\|_{n,\phi}$ on $\mathcal M(\cK\otimes A)p_n$ via 
	\[
	\|x\|_{n,\phi}=(\tau_{m_n}\otimes\phi)(x^*x)^{1/2}.
	\]
Using Lemma~\ref{L.state-in-M(A)}, it follows for any $x\in\mathcal M(\cK\otimes A)p_n$ that
	\begin{equation} \label{eq:norm-supremum}
		\|x\|_{m_n,A}^2 = \|\cE_{m_n}(x^*x)\| = \sup_{\phi\in S(A)} \phi(\cE_{m_n}(x^*x)) = \sup_{\phi\in S(A)} \|x\|_{n,\phi}^2.
	\end{equation}
	For the rest of the proof we fix some $\phi\in S(A)$.
	Under the identification $\Phi$ above, for any $X=\sum_{j=1}^k Z_j\otimes x_j$ in $\cB( H_n,K)\otimes \mathcal M(A)$ we have that
	\[
	\Phi(X)^*\Phi(X)=\sum_{i,j=1}^k (Z_i^* Z_j)\otimes (x_i^*x_j) \ \in \ p_n\mathcal M(\cK\otimes A)p_n, 
	\]
	and hence
	\begin{equation} \label{eq:Hphi-norm}
		(\tau_{m_n}\otimes\phi)(\Phi(X)^*\Phi(X)) = \sum_{i,j=1}^k \tau_{m_n}(Z_i^*Z_j) \phi(x_i^*x_j).
	\end{equation}
	Let $H_\phi$ be the Hilbert space obtained from $(A,\phi)$ via the GNS-construction.
	If we view $\cB(H_n,K)$ as a Hilbert space with the normalized Hilbert--Schmidt norm as before, then we have a canonical contractive linear map ($\hat\otimes$ is the Hilbert space tensor product)
	\[
	\Theta^\phi: \cB(H_n,K)\otimes\mathcal M(A) \to  \cB(H_n,K)\hat{\otimes} H_\phi.
	\]
	Considering how the norm on the tensor product Hilbert space is defined, condition \eqref{eq:Hphi-norm} directly shows
	\[
	\|\Phi(X)\|_{n,\phi}=\|\Theta^{\phi}(X)\| \quad\text{for all } X\in \cB( H_n,K)\otimes_{\mathbb{C}} \mathcal M(A),
	\]
	or equivalently, since $\Phi$ is an isomorphism, 
	\begin{equation} \label{eq:norm-compatibility}
		\|(\Theta^\phi\circ\Phi^{-1})(x)\| = \|x\|_{n,\phi} \quad\text{for all } x\in P_N \mathcal M(\cK\otimes A)p_n.
	\end{equation}
	If we restrict $(\Theta^\phi\circ\Phi^{-1})$ to $p_n\mathcal M(\cK\otimes A)p_n$, then we can observe that the diagram
	\begin{equation} \label{eq:the-diagram}
		\xymatrix{
			p_n \mathcal M(\cK\otimes A)p_n \ar[d]^{\cE_{m_n}} \ar[r]^{\Phi^{-1}} &\cB(H_n,H_n)\otimes\mathcal M(A)\ar[r]^{\Theta^\phi} & \cB(H_n,H_n) \hat{\otimes} H_\phi \ar[d]^{\tau_{m_n}\otimes 1_H} \\
			p_n\otimes\mathcal M(A) \ar[rr] && \bbC p_n\hat{\otimes} H_\phi
		}
	\end{equation}
commutes, where the lower horizontal arrow is the linear contractive map induced from the GNS-construction.
We note furthermore that the map $\Theta^\phi\circ\Phi^{-1}$ is $\rho$-equivariant in the obvious sense, putting us in the position to apply Lemma \ref{L.Schur}.
We consider the unitary representation 
\[
	\sigma_{n,\phi}: \Gamma\to\mathcal U(\cB(H_n,K)\hat{\otimes} H_\phi),\quad \sigma_{n,\phi}(g)(Z\otimes\xi):=(\rho(g)Z\rho_n(g)^*)\otimes\xi
\]
for all $g\in\Gamma$, $Z\in\cB(H_n,K)$ and $\xi\in H_\phi$.
Set $p_n^\perp=P_N-p_n$ and consider 
\begin{align*}
\eta_1:=(\Theta^\phi\circ\Phi^{-1})(p_nT) \text{ and } \eta_2:=(\Theta^\phi\circ\Phi^{-1})(p_n^\perp T). 
\end{align*} 
We get the inequality 
\[
	\max_{g\in F} \max_{q\in\{p_n, p_n^\perp\}}\| q\rho(g)T\rho_n(g)^* - qT \|_{n,\phi} \leq \max_{g\in F} \| \rho(g)T\rho(g)^*-T \|<\varepsilon\delta
\]
which via \eqref{eq:norm-compatibility} directly translates   to
\begin{equation} \label{eq:eta-invariance}
	\max_{g\in F} \ \max_{j=1,2} \| \sigma_{n,\phi}(g)\eta_j - \eta_j \|<\varepsilon\delta.
\end{equation}
On the one hand, we have $\eta_1\in\cB(H_n,H_n)\hat{\otimes} H_\phi$ and $\sigma_{n,\phi}$ is of the form as given in Lemma \ref{L.Schur}\ref{1.Schur}.
Thus
\[
	\|\eta_1-(\tau_{m_n}\otimes 1_{H_\phi})\eta_1\|<\delta.
\]
In light of the commuting diagram \eqref{eq:the-diagram}, this is the same as
\[
	\|p_nT-\cE_{m_n}(p_nT)\|_{n,\phi}<\delta.
\]
On the other hand, we have $\eta_2\in\cB(H_n,K\ominus H_n)$ and on this space $\sigma_{n,\phi}$ is of the form as given in Lemma \ref{L.Schur}\ref{2.Schur}.
This lemma, together with \eqref{eq:eta-invariance}, implies $\|\eta_2\|<\delta$, which with \eqref{eq:norm-compatibility} implies $\|p_n^{\perp} T\|_{n,\phi}<\delta$.
All in all we obtain 
\[
	\|T-p_n\otimes\cE_{m_n}(p_nT)\|_{n,\phi}<2\delta.
\]
Since the state $\phi$ was arbitrarily chosen, condition \eqref{eq:norm-supremum} yields 
\[
	\|T-p_n\otimes\cE_{m_n}(p_nT)\|_{m_n,A} \leq 2\delta
\]
and proves the claim.
\end{proof}

\section{Property (T) groups inside stable coronas} \label{S.Gamma.A}

The main result of this section is Theorem~\ref{T.Psi+}.
For any non-zero \cstar-algebra $A$, it implies a transfer result for (relatively) approximately central sequences from $\cQ(\cK\otimes A)$ to $\cM(A)$.
As a consequence we obtain Corollary~\ref{C.Psi}, which represents the most difficult implication in Theorem~\ref{T.main}.

The proof of Theorem~\ref{T.Psi+} is based on the reinterpretation of Wassermann’s proof that if $\Gamma$ is a property (T) group with infinitely many finite-dimensional representations, then $\Ext(\cstar(\Gamma))$ is not a group (\cite{wassermann1991c}) given in \cite[\S 2]{farah2023calkin}, and some overlap with the text of the latter is included for the reader’s convenience. 

\begin{notation}
Suppose $\Gamma$ is a property (T) group with infinitely many inequivalent irreducible representations on finite-di\-men\-sio\-nal Hilbert spaces (e.g., a residually finite property (T) group such as $\SL_3(\bbZ)$).
Fix a Kazhdan pair,  $F\Subset \Gamma$ and $\varepsilon>0$.
Let $\rho_n$, for $n\in \bbN$, be an enumeration of irreducible representations of $\Gamma$ on finite-dimensional Hilbert spaces.
Since there are only finitely many inequivalent representations of $\Gamma$ on every fixed finite-dimensional Hilbert space $H$ (\cite[Corollary~3]{wassermann1991c}), we may choose the enumeration so that the sequence $m_n:=\dim(H_n)$ is nondecreasing, and necessarily $m_n\to \infty$ as $n\to \infty$.
For each $n\geq 1$, let $p_n$ be the Kazhdan projection associated with $\rho_n$ in $\cstar(\Gamma)$. 
Each representation $\rho_n$ of $\Gamma$ on $H_n$ uniquely extends to a representation of the full group algebra  $\cstar(\Gamma)$ on the same space.
We will slightly abuse the notation and denote the latter by~$\rho_n$.  
Let $\rho=\bigoplus_n \rho_n$, and $H:=\bigoplus_n H_n$.
Then $H$ is an infinite-dimensional, separable Hilbert space and the \cstar-algebra generated by $\rho(\Gamma)$ contains all~$p_n$ (the orthogonal projection onto $H_n$), for $n\in \bbN$.
\end{notation} 

\begin{definition} \label{D.W(Gamma)}
Let $\Gamma$ be a countable group with property (T) and with infinitely many irreducible representations $\rho_n\colon \Gamma\to \cB(H_n)$, for $n\in \bbN$, on finite-dimensional Hilbert spaces.
Let $\rho=\bigoplus_n \rho_n$ be the direct sum representation on $H=\bigoplus_n H_n$.
Given a non-zero \cstar-algebra $A$, as in \eqref{eq.B(H)-inside-M(KotimesA)} we identify $\cB(H)$ with a subalgebra of  $\mathcal M(\cK\otimes A)$.
In this way we can define the unitary representation $\bar{\rho}: \Gamma\to\mathcal U(\mathcal Q(\cK\otimes A))$ via 
\[
\bar{\rho}(g)=(\rho(g)\otimes 1) + \cK\otimes A
\]
for all $g\in\Gamma$.
\end{definition}

\begin{definition} \label{def:the-map}
Adopt the notation of Definition \ref{D.W(Gamma)} and write $p_n$ for the orthogonal projection onto $H_n$.
We define an embedding 
\[
\iota: \ell^\infty(\mathbb N,\mathcal M(A))\to\mathcal M(\cK\otimes A) \text{ via } \iota((x_n)_n)=\sum_{n=1}^\infty p_n\otimes x_n,
\]
where the limit of partial sums is taken in the strict topology.
We consider a conditional expectation $\cE$ from $\cM(\cK\otimes A)$ onto the image of $\iota$ by (using the conditional expectation $\cE_{m_n}\colon M_{m_n}(\cM(A))\to\cM(A)$ as in \eqref{eq.Em}):
\[
\cE(x)=\sum_{n=1}^\infty p_n\otimes \cE_{m_n}(p_nxp_n),
\]
where the limit is again taken in the strict topology.   
We observe that for $x\in\cK\otimes A$, the sequence $p_n x p_n\in A$ converges to zero in norm, hence $\cE(\cK\otimes A)=\iota(c_0(A))$.
This gives rise to a unital completely positive map $\Psi: \cQ(\cK\otimes A)\to\cQ(c_0(A))=\ell^\infty(\mathbb N,\mathcal M(A))/c_0(\bbN, A)$ defined by 
\begin{equation}\label{eq.Psi}
	\Psi(x+\cK\otimes A)=(\iota^{-1}\circ\cE)(x)+c_0(A).
\end{equation}
By $\bar{\Psi}: \cQ(\cK\otimes A)\to\cM(A)_\infty$ we denote  the composition of $\Psi$ with the quotient map $\cQ(c_0(A))\to\cM(A)_\infty$.
(Note that $\Psi=\bar{\Psi}$ if $A$ is unital.)
\end{definition}

\begin{lemma} \label{lem:multiplicativity} Assume $A$ is a non-zero \cstar-algebra and adopt the notation of both Definitions \ref{D.W(Gamma)} and \ref{def:the-map}.
Suppose that $(F,\varepsilon)$ is a Kazhdan pair for $\Gamma$.
Then it follows for every $z\in\cQ(\cK\otimes A)$ that
\[
\|\bar{\Psi}(z^*z)-\bar{\Psi}(z)^*\bar{\Psi}(z)\|\leq 8\varepsilon^{-1}\|z\|\max_{g\in F}\|\bar{\rho}(g)z-z\bar{\rho}(g)\|.
\]
In particular, the relative commutant $\cQ(\cK\otimes A)\cap\{\bar{\rho}(g)\}_{g\in\Gamma}'$ belongs to the multiplicative domain of $\bar{\Psi}$.
\end{lemma}
\begin{proof}
Let $\eta>0$.
Let $T\in\mathcal M(\cK\otimes A)$ be a multiplier with $z=T+\cK\otimes A$ and $\|T\|=\|z\|$. Choose $\delta>0$ such that
\[
\varepsilon\delta > \max_{g\in F}\|\bar{\rho}(g)z-z\bar{\rho}(g)\|.
\]
Since the sequence of projections $p_n$ converges to zero strictly and these projections commute with $\rho$ on the nose, we obtain
\[
\limsup_{n\to\infty} \| \big(\rho(g)T\rho(g)^*-T\big)p_n\| < \varepsilon\delta.
\]
By Lemma \ref{L.Schur.spicy.version}, we get 
\[
\limsup_{n\to\infty} \|Tp_n-p_n\otimes \cE_{m_n}(p_nTp_n)\|_{m_n,A} \leq 2\delta.
\]
The same holds if we insert $T^*$ in place of $T$.
Since we have
\[
\|\bar{\rho}(g)z^*z-z^*z\bar{\rho}(g)\|\leq 2\|z\| \|\bar{\rho}(g)z-z\bar{\rho}(g)\|,\quad g\in F,
\]
the analogous argument also shows
\[
\limsup_{n\to\infty} \|T^*Tp_n-p_n\otimes \cE_{m_n}(p_nT^*Tp_n)\|_{m_n,A} \leq 4\|z\|\delta.
\]
Using the elementary properties of the norm $\|\cdot\|_{m,A}$ stated in  \eqref{eq:generalized-HS-equality} and \eqref{eq:generalized-HS-inequalities}, we can compute
\[
\begin{array}{cl}
\multicolumn{2}{l}{ \|\cE_{m_n}(p_n T^*T p_n)-\cE_{m_n}(p_nT^*p_n)\cE_{m_n}(p_nTp_n)\| }\\
=& \|p_n\otimes\cE_{m_n}(p_n T^*T p_n)-p_n\otimes\cE_{m_n}(p_nT^*p_n)\cE_{m_n}(p_nTp_n)\|_{m_n,A} \\
=& \|p_n\otimes\cE_{m_n}(p_n T^*T p_n) - T^*Tp_n\|_{m_n,A} \\
&+ \| T^*Tp_n-T^*(p_n\otimes\cE_{m_n}(p_nTp_n)) \|_{m_n,A} \\
&+ \| T^*(p_n\otimes\cE_{m_n}(p_nTp_n)) - p_n\otimes\cE_{m_n}(p_nT^*p_n)\cE_{m_n}(p_nTp_n) \|_{m_n,A} \\
=& \|p_n\otimes\cE_{m_n}(p_n T^*T p_n) - T^*Tp_n\|_{m_n,A} \\
&+ \|T\|\cdot \| Tp_n-p_n\otimes\cE_{m_n}(p_nTp_n) \|_{m_n,A} \\
&+ \underbrace{\|\cE_{m_n}(p_nTp_n) \|}_{\leq\|T\|} \cdot \| T^*p_n - p_n\otimes\cE_{m_n}(p_nT^*p_n) \|_{m_n,A}
\end{array}
\]
From our earlier observations and $\|z\|=\|T\|$, this combines into
\begin{align*}
\|\bar{\Psi}(z^*z)-&\bar{\Psi}(z)^*\bar{\Psi}(z)\|  \\
&= \big\| \iota^{-1}\big( \cE(T^*T)-\cE(T)^*\cE(T) \big) + c_0(\cM(A)) \big\| \\
&= \displaystyle \limsup_{n\to\infty} \|\cE_{m_n}(p_n T^*T p_n)-\cE_{m_n}(p_nT^*p_n)\cE_{m_n}(p_nTp_n)\| \\
&\leq 8\|z\|\delta.
\end{align*}
Since we chose $\delta>\varepsilon^{-1}\max_{g\in F}\|\bar{\rho}(g)z-z\bar{\rho}(g)\|$ arbitrarily, this proves the claim.
\end{proof}

As a consequence, we can prove the following generalization of \cite[Proposition~2.3]{farah2023calkin}. 

\begin{theorem} \label{T.Psi+}
Suppose that $A$ is a nonzero \cstar-algebra.
Denote by $\pi: \cM(\cK\otimes A)\to\cQ(\cK\otimes A)$ the quotient map.
Then there exist:
	\begin{enumerate}[label=\textup{(\arabic*)}]
	\item a unital $*$-embedding $\Theta\colon \cQ(c_0(A))\to \cQ(\cK\otimes A)$ with $\Theta(\cM(A))=\pi(1\otimes\cM(A))$.
	\item a unital completely positive map $\Psi: \cQ(\cK\otimes A)\to \cQ(c_0(A))$ with $\Psi\circ\Theta=\id$.
	\item a separable unital \cstar-subalgebra $C$ of $\pi(\cM(\cK)\otimes 1)\cap\Theta(\cQ(c_0(A)))'$ (and hence of $\cQ(\cK\otimes A)$) such that after composing with the quotient map $\cQ(c_0(A))\to\cM(A)_\infty$, the restriction
	\[
	\bar{\Psi}: \cQ(\cK\otimes A)\cap C'\to\cM(A)_\infty
	\]
	as well as the restriction of the induced sequence algebra map\footnote{Again $C$ is identified with a \cstar-subalgebra of the diagonal copy of $\cQ(\cK\otimes A)$ inside $\cQ(\cK\otimes A)_\infty$.}
	\[
	\bar{\Psi}_\infty: \cQ(\cK\otimes A)_\infty\cap C'\to (\cM(A)_\infty)_\infty
	\]
	are $*$-homomorphisms.
	\end{enumerate} 
\end{theorem}

\begin{tikzpicture}
	\matrix[row sep=15mm,column sep=20mm] {
		\node  (C') {$\cQ(\cK\otimes A)\cap C'$};
		& \node (QK) {$\cQ(\cK\otimes A)$};
		& \node  (Qc0) {$\cQ(c_0(A))$};\\
		\node (C'inf) {$\cQ(\cK\otimes A)_\infty\cap C'$};
		& \node (MAinf) {$(\cM(A)_\infty)_\infty$};
		&\node (MA) {$\cM(A)_\infty$};
		\\
	};
	\draw (Qc0) edge[->>] node [left] {}  (MA) ;
	\path (Qc0)  edge[->, bend left]  node [above] {$\Theta$} (QK) ;
	\draw (QK) edge[->] node [above] {$\Psi$} (Qc0);
	\draw (C') edge [->] node [above] {$\subseteq$} (QK);
	\draw (C') edge [->] node [left] {\rotatebox[origin=c]{-90}{$\subseteq$}} (C'inf);
	\draw (MA) edge [->] node [above] {$\supseteq$} (MAinf);
	\draw (C') edge[->]  node [above] {$\bar\Psi$} (MA);
	\draw (C'inf) edge[->]  node [above] {$\bar{\Psi}_\infty$} (MAinf);
\end{tikzpicture}

\begin{proof} 
Choose a residually finite group $\Gamma$ with property (T) and construct the unitary representation $\bar{\rho}: \Gamma\to\mathcal U(\cQ(\cK\otimes A))$ as in Definition \ref{D.W(Gamma)}. It factors through $\pi(\cM(\cK)\otimes 1)$ by construction. 
Let $C:=\cstar(\bar{\rho}[\Gamma])$. 
	We consider the unital, completely positive, map $\Psi: \cQ(\cK\otimes A)\to\cQ(c_0(A))$ from \eqref{eq.Psi} of Definition~\ref{def:the-map}. 
	With $p_n$, for $n\in \bbN$, denoting the images of Kazhdan projections as in Definition~\ref{def:the-map}, we identify $\cQ(c_0(A))$ with 
	\[
	\prod_n (p_n\otimes \cM(A))/\bigoplus_n (p_n \otimes A).
	\]
This is a subalgebra of $\cQ(\cK\otimes A)$; the map $\Theta$ from the statement of the theorem is this identification.
The formula $\Psi\circ\Theta=\id$ is then evident from construction.
Clearly the range of $\Theta$ commutes pointwise with $C$.
It only remains to prove that the maps  $\bar{\Psi}$ and $\bar{\Psi}_\infty$ are multiplicative.  Equivalently, we need to prove that their restricted domains are included in their multiplicative domains. 
Since the first map can be seen as a restriction (and corestriction) of the second, it suffices to treat the latter.
	
Fix $y\in\cQ(\cK\otimes A)_\infty\cap C'$.
We lift $y$ to a bounded sequence $(y_n)_n$ in $\cQ(\cK\otimes A)$, which necessarily satisfies $\|[y_n,c]\|\to 0$ for all $c\in C$.
With $(F,\varepsilon)$ being the Kazhdan pair for $\Gamma$, Lemma~\ref{lem:multiplicativity} implies that 
	\begin{align*}
\|\bar{\Psi}_\infty(y^*y)-\bar{\Psi}_\infty(y)^*&\bar{\Psi}_\infty(y)\| \\ 
	=& \displaystyle \limsup_{n\to\infty} \| \bar{\Psi}(y_n^*y_n)- \bar{\Psi}(y_n)^* \bar{\Psi}(y_n) \| \\
	\leq& \displaystyle \limsup_{n\to\infty} 8\varepsilon^{-1}\|y_n\|\max_{g\in F}\|\bar{\rho}(g)y_n-y_n\bar{\rho}(g)\| \ = \ 0.
\end{align*}
This finishes the proof
\end{proof}
	
For the next statement, recall from Definition~\ref{Def.approx-central-embeddings} the notion of large approximately central maps. 
	
\begin{corollary} \label{C.Psi} 
Suppose that $A$ is a nonzero \cstar-algebra and $B$ is a separable, unital \cstar-algebra.
If $\cQ(\cK\otimes A)$ admits large approximately central maps from~$B$, then so does $\cM(A)$. 
In particular, if $\calD$ is a strongly self-absorbing \cstar-algebra and $\cQ(\cK\otimes A)$ is separably $\calD$-stable, then so is $\cM(A)$.
\end{corollary}	
\begin{proof}
Let $\pi:\cM(\cK\otimes A)\to\cQ(\cK\otimes A)$ be the quotient map.
We apply Theorem~\ref{T.Psi+} and choose $\Theta$, $\Psi$, $\bar{\Psi}$, $\bar{\Psi}_\infty$,  and $C$ as provided by this theorem.
Let $S\subset\cM(A)$ be a separable \cstar-subalgebra.
By assumption, we may find a unital $*$-homomorphism from $B$ to $\cQ(\cK\otimes A)_\infty\cap C'\cap \pi(1\otimes S)'$.
By composing with the $*$-homomorphism $\bar{\Psi}_\infty: \cQ(\cK\otimes A)_\infty\cap C' \to (\cM(A)_\infty)_\infty$ and keeping in mind that $\bar{\Psi}_\infty$ sends $\pi(1\otimes S)$ to the diagonal image of $S$, we obtain a unital $*$-homomorphism
\[
B \to (\cM(A)_\infty)_\infty\cap S'.
\]
Since $(\cM(A)_\infty)_\infty$ is canonically isomorphic to $\cM(A)_\cF$ for the filter 
\[
\cF:=\{X\subseteq \bbN\times\bbN\mid (\forall^\infty m)(\forall^\infty n)(m,n)\in X\}
\]
 (in a way that preserves the diagonal copy of $\cM(A)$)  and since $S$ was arbitrary, it follows from Proposition~\ref{P.approx-central-embeddings} that $\cM(A)$ admits large approximately central maps from $B$.

The ``in particular'' part is a direct consequence of Corollary~\ref{cor:axiomatizable}.
\end{proof}

We are finally ready to combine all of our main results proven thus far:

\begin{proof}[Proof of Theorem~\ref{T.main}]
\ref{T.main.A-tensorially-absorbing}$\Rightarrow$\ref{T.main.Q-locally-absorbing} is a direct consequence of Proposition~\ref{P.sep-D-is-stable}\ref{P.sep-D-is-stable:2} and Theorem~\ref{T.Q-saturated.1}.
The implication \ref{T.main.Q-locally-absorbing}$\Rightarrow$\ref{T.main.M-locally-absorbing} is Corollary~\ref{C.Psi}.
The implication \ref{T.main.M-locally-absorbing}$\Rightarrow$\ref{T.main.A-tensorially-absorbing} is Proposition~\ref{P.A-in-M(A)}.
\end{proof}

It is not in general true that for any non-unital $\sigma$-unital \cstar-algebra $A$, separable $\calD$-stability of $\cQ(A)$ forces $A$ to be separably $\calD$-stable. 
For an easy counterexample, take for instance $A:=\bbC\oplus (\cK\otimes \calD)$.

\begin{remark}
If one considers the ucp map $\Psi$ from Theorem~\ref{T.Psi+}, then it is possible to show that the restriction $\Psi: \cQ(\cK\otimes A)\cap C'\to\cQ(c_0(A))$ is already a $*$-homomorphism, without having to project down onto $\cM(A)_\infty$. This requires a bit of extra work and we decided not to include the argument because, while being non-trivial, we found no interesting applications of this fact that are not already covered by what we proved above.
Furthermore, said argument does not seem to allow one to conclude (for non-unital $A$) that the induced sequence algebra map $\Psi_\infty: \cQ(\cK\otimes A)_\infty\cap C'\to\cQ(c_0(A))_\infty$ is multiplicative; in fact we would not be surprised if this could be wrong in general.
\end{remark}


\section{Applications and concluding remarks} \label{S.Applications}

The Calkin algebra was considered a prime candidate for a \cstar-algebra not elementarily equivalent to a nuclear \cstar-algebra, but previous attempts at a proof failed and in the meantime other examples have been found (\cite [\S 7]{Muenster}). (Two \cstar-algebras are \emph{elementarily equivalent} if and only if they have the same first-order theory.)
The following observation was communicated to us by James Gabe, Chris Schafhauser, and Stuart White during the Memorial Conference in honour of Eberhard Kirchberg in M\"unster in July 2023.  It precipitated the work presented here and is  included with their kind permission. 

\begin{theorem}\label{T.Calkin} The Calkin algebra $\cQ(H)$ is not elementarily equivalent to a nuclear \cstar-algebra. 
\end{theorem}

\begin{proof}
Suppose otherwise and fix a separable, nuclear \cstar-algebra $A$ elementarily equivalent to $\cQ(H)$.
The Calkin algebra is purely infinite and simple. This is an axiomatizable property (\cite[Theorem~2.1]{Muenster}), and therefore~$A$ is purely infinite and simple.
By a result of Kirchberg, $A$ tensorially absorbs~$\cO_\infty$ (\cite[Theorem~7.2.6]{Ror:Classification}).
By Corollary~\ref{cor:axiomatizable}, $\cQ(H)$ is separably $\cO_\infty$-stable, which contradicts Theorem~\ref{T.main}.
(With a few extra arguments it would have been also possible to reach a contradiction to \cite[Theorem~B]{farah2023calkin}.) 
\end{proof}

In case when $\calD$ is the Jiang--Su algebra $\cZ$, Theorem~\ref{T.main} shows that if $A$ is $\sigma$-unital, then $A$ is separably $\cZ$-stable if and only if $\cM(A)$ is separably $\cZ$-stable.
As Hannes Thiel pointed out to us, this is closely related to \cite[Theorem~5.3]{kaftal2017strict} where it was proven that if $A$ is $\sigma$-unital and has strict comparison,   then $\cM(A)$ has strict comparison, under the additional assumption that $A$ has finitely many extremal tracial states. 
Since every $\cZ$-absorbing \cstar-algebra has strict comparison (\cite[Corollary~4.6]{rordam2004stable}), we obtain: 

\begin{corollary} \label{C.strict}
If $A$ is $\sigma$-unital and separably $\cZ$-stable, then $\cM(A)$ has strict comparison.
\end{corollary}

Modulo verification of the Toms--Winter conjecture (for up-to-date information see e.g., \cite[\S 1.6]{carrion2023tracially}), Theorem~\ref{T.main} would hence imply that if $A$ is separable simple nuclear, non-elementary and has strict comparison, then $\cM(A)$ has strict comparison.

The original motivation for the earlier article \cite{farah2023calkin} came from the rigidity question for coronas of separable, non-unital \cstar-algebras, asking when $\cQ(A)\cong \cQ(B)$ implies $A\cong B$ attributed to Sakai in \cite{Ell:Derivations}  (see \cite[\S 4.4]{farah2022corona} for more information). 
Our Theorem~\ref{T.main} gives a partial positive answer to  this question. 

\begin{corollary} If $A$ and $B$ are separable \cstar-algebras such that there exists a strongly self-absorbing \cstar-algebra $\calD$ that is tensorially absorbed by $A$ but not by $B$, then 
$\cQ(\cK\otimes A)$ and $\cQ(\cK\otimes B)$ are not isomorphic. 

In particular, coronas of stabilizations of strongly self-absorbing \cstar-al\-geb\-ras are pairwise nonisomorphic. 
\end{corollary}  

\begin{corollary}\label{C.JiangSu}
Assume $A$ is separable, non-unital and $\cQ(A)$ is isomorphic to the Calkin algebra.
Then $A$ does not absorb the Jiang--Su algebra tensorially.  
\end{corollary}
\begin{proof}
If $A$ is $\cZ$-absorbing, then its corona is $\cZ$-saturated by Theorem~\ref{T.Q-saturated}.
This is not the case with the Calkin algebra by \cite[Theorem~B]{farah2023calkin}. 
\end{proof}

We now turn to an application of Theorem~\ref{T.Q-saturated} related to classification theory.
In the recent breakthrough paper \cite{carrion2023class}, a new conceptual proof is given of the state-of-the-art classification theorem for separable unital simple nuclear $\cZ$-stable \cstar-algebras satisfying the UCT.
The approach towards such a big result involved various independent steps, both conceptual and technical.
In order to get to the uniqueness theorem in that article, the authors re-examined the stable uniqueness theorem for $\textit{KK}$-theory and proved a Jiang--Su stable counterpart in the form of \cite[Theorems 1.4, 5.15]{carrion2023class}.
We note that it is outside the scope of this article to recall all the needed terminology and the framework to appreciate these statements related to $\textit{KK}$- and $\textit{KL}$-theory, but we refer the reader to said article for a very thorough treatment and its applications.
We can improve the conclusion of the $\cZ$-stable $\textit{KK}$-uniqueness theorem from \cite{carrion2023class} as follows:

\begin{theorem} \label{T.Z-stable-KK}
Let $A$ be a separable \cstar-algebra.
Let $I$ be a stable, $\sigma$-unital and separably $\cZ$-stable \cstar-algebra.
Let $\phi,\psi: A\to\cM(I)$ be a pair of absorbing $*$-homomorphisms that form a Cuntz pair.
	\begin{enumerate}[label=\textup{(\arabic*)}]
	\item If $[\phi,\psi]_{\textit{KK}(A,I)}=0$, then there exists a norm-continuous path $(u_t)_{t\geq 0}$ of unitaries in $I^\dagger$ such that\footnote{Here $I^\dagger$ denotes the smallest unitization of $I$.}
	\[
	\psi(a)=\lim_{t\to\infty} u_t\phi(a)u_t^*,\quad a\in A.
	\]
	\item If $[\phi,\psi]_{\textit{KL}(A,I)}=0$, then there exists a sequence $(u_n)_{n\in\bbN}$ of unitaries in $I^\dagger$ such that
	\[
	\psi(a)=\lim_{n\to\infty} u_n\phi(a)u_n^*,\quad a\in A.
	\]
	\end{enumerate}
\end{theorem}
\begin{proof}
Let $\pi: \cM(I)\to\cQ(I)$ be the quotient map.
By Theorem~\ref{T.Q-saturated}, we have that $\cQ(I)$ is $\cZ$-saturated.
Therefore $\cQ(I)\cap\pi(\phi(A))'$ is also $\cZ$-saturated.
By \cite[Theorem 4.8]{carrion2023class} (which goes back to \cite{Jiang99}), it follows that $\cQ(I)\cap\pi(\phi(A))'$ is $K_1$-injective.
By the explicit observations and remarks made right after \cite[Question 5.17]{carrion2023class}, the desired conclusion follows.
\end{proof}

By using this strengthened version of uniqueness theorem, it is possible to simplify or circumvent various technical steps in \cite[Section~7]{carrion2023class}.

\begin{remark}
	The main results of this paper have been subsequently generalized to the dynamical setting in \cite{li2025corona}.
	This arises by letting a locally compact group $G$ act on a \cstar-algebra $A$ and canonically inducing a $G$-action on $\cQ(A)$.
	The core problems studied in this article have a dynamical analog if strong self-absorption of \cstar-algebras is replaced with its dynamical counterpart from \cite{szabo2018strongly}. 
\end{remark}

\bibliographystyle{plain}
\bibliography{references}

\end{document}